\documentclass[a4paper]{article}

\usepackage[bottom=2in]{geometry} % full-width

\usepackage{amsmath}
\usepackage{amsthm}
\usepackage{amssymb}
\usepackage{color,todonotes}
\usepackage[utf8]{inputenc}
\usepackage{hyperref}
\hypersetup{
	unicode,
	pdfauthor={Author One, Author Two, Author Three},
	pdftitle={A simple article template},
	pdfsubject={A simple article template},
	pdfkeywords={article, template, simple},
	pdfproducer={LaTeX},
	pdfcreator={pdflatex}
}

\usepackage{graphicx}
\usepackage{subcaption}

\usepackage[sort&compress,numbers,square]{natbib}
\theoremstyle{plain}
\newtheorem{theorem}{Theorem}[section] %like this we have Theorem 3.2
\newtheorem{corollary}[theorem]{Corollary}
\newtheorem{lemma}[theorem]{Lemma}

\newtheorem{remark}[theorem]{Remark}

\newtheorem{definition}[theorem]{Definition}
\newtheorem{example}[theorem]{Example}

\usepackage{booktabs}

\usepackage{multirow}

\usepackage[normalem]{ulem}

\usepackage{graphicx, color}
\graphicspath{{fig/}}

\usepackage{algorithm, algpseudocode} % use algorithm and algorithmicx for typesetting algorithms
\usepackage{mathrsfs} % for \mathscr command
\usepackage{bm}

\def\Ab{{\mathsf A}}
\def\Bb{{\mathsf B}}

\def\Eb{{\mathsf E}}
\def\Fb{{\mathsf F}}
\def\Gb{{\mathsf G}}

\def\Ib{{\mathsf I}}

\def\Nb{{\mathsf N}}
\def\Ob{{\mathsf O}}

\def\Qb{{\mathsf Q}}

\def\Sb{{\mathsf S}}

\def\Ub{{\mathsf U}}

\def\Wb{{\mathsf W}}

\def\Yb{{\mathsf Y}}

\def\T{^{\top}}

\def\aa{{\mathbf a}}

\def\cc{{\mathbf c}}

\def\e{{\mathbf e}}

\def\s{{\mathbf s}}

\def\u{{\mathbf u}}
\def\v{{\mathbf v}}

\def\x{{\mathbf x}}
\def\y{{\mathbf y}}
\def\z{{\mathbf z}}

\def\sSigma{{\bm \Sigma}}
\def\svec{\operatorname{svec}}
\def\PP{{\mathbb P}}

\newcommand{\norm}[1]{\left\lVert#1\right\rVert}

\newcommand{\code}[1]{\textup{\texttt{#1}}}

\usepackage{enumitem}% roman

\usepackage{bbm}

\iftrue
\newcommand{\comments}[1]{\footnote{\textcolor{blue}{\textit{#1}}}}

\else
\newcommand{\comments}[1]{}

\fi

\parskip=\medskipamount

\title{Distributionally robust chance-constrained standard quadratic optimization with Wasserstein ambiguity}
     
\author{Immanuel M. Bomze\thanks{Research Network Data Science and Faculty of Mathematics, University of Vienna, Oskar-Morgenstern-Platz 1, A-1090 Vienna, Austria. E-mail: immanuel.bomze@univie.ac.at} \and Daniel De Vicente\thanks{University of Vienna, Oskar-Morgenstern-Platz 1, A-1090 Vienna, Austria. E-mail: a12032143@unet.univie.ac.at }\and Simge Küçükyavuz \thanks{Department of Industrial Engineering and Management Sciences, Northwestern University, Evanston, IL 60208, USA. E-mail: simge@northwestern.edu}
}

\date{November 14, 2026}

\begin{document}
\maketitle
\begin{abstract} 
 \sloppy{The standard quadratic optimization problem consists of minimizing a quadratic form over the standard simplex. It has highly relevant applications ranging from portfolio optimization to machine learning. However, if the quadratic form is neither convex nor concave, this problem is NP-hard. We study a distributionally robust chance-constrained variant whose ambiguity set is a Wasserstein ball around the sampled distribution and make four contributions. First, although the problem is nonconvex in the decision variable, we exploit its linear dependence on the uncertain matrix to reduce the problem to a mixed-integer nonlinear optimization problem. Second, we show that a judicious choice of matrix norm---the entrywise maximum norm---makes this reformulation both exact and tractable, and we describe the induced worst-case distribution. Third, our formulation is, by construction, the minimization of the worst-case Value-at-Risk of the random quadratic form; this unifies it with existing frameworks for uncertain standard quadratic optimization problems, in particular under chance constraints, and we identify natural settings in which it collapses to a single deterministic instance. Finally, we report numerical experiments on positive definite and indefinite  instances which quantify how the model parameters drive the optimal value and solution time, and expose the computational cost of indefiniteness.}
	\end{abstract}

\tableofcontents

%----------------------------------
%       1. Introduction 
% ---------------------------------

\section{Introduction}\label{sec:1.intro}

\subsection{Motivation and historical remarks}

\noindent The standard quadratic optimization problem~\eqref{StQP} consists of minimizing a quadratic form over the standard simplex
\begin{equation}\label{StQP}\tag{StQP}
\upsilon(\Qb):=\min_{\x \in \Delta} \,\x^{\top}\Qb\x,
\end{equation}
where $\Delta:= \{\x \in \mathbb{R}^n_+: \e^{\top}\x = 1\}$ is the standard simplex in $\mathbb{R}^n$ and $\Qb$ is a symmetric $n\times n$ matrix (denoted by $\Qb\in \mathcal S^n$ in the sequel). Here $\e\in \mathbb{R}^n$ is the vector of all ones and ${}^{\top}$ denotes transposition.  The objective function is already in general form since any general quadratic objective function $q(\x):=\x^{\top}\Ab\x + 2\cc^{\top}\x$ can be written in homogeneous form  $q(\x):= \x^{\top}\Qb\x$ by defining  $\Qb:= \Ab + \cc\e^{\top} + \e\cc^{\top}\in \mathcal S^n$.

\noindent Even though the StQP seems rather easy—minimization of a quadratic function under linear constraints—it is NP-hard without assumptions on the definiteness of the matrix $\Qb$. %Motzkin \& Straus 
{In~\cite{motzkin-straus-1965}, the authors} showed that the maximum clique problem, a well-known NP-hard problem, can be reduced to an StQP. Hence, the StQP is often regarded as one of the simplest hard problems~\cite{bomze2018complexity} because it has the simplest nonconvex objective, a quadratic form, and the simplest feasible set: a polytope. The StQP is a very flexible optimization class that allows for modelling of diverse problems such as portfolio optimization problems~\cite{markowitz-1952}, pairwise clustering~\cite{pavan2003dominant} and replicator dynamics~\cite{bomze1998standard}.

\noindent The only data required to fully characterize an StQP is the data matrix $\Qb$. However, in many applications the matrix $\Qb$ is uncertain. StQPs with uncertain data have been explored in the literature.

\noindent One of the most natural ways to handle uncertain objective functions is the worst-case approach, which assumes no information about the distribution of the data matrix $\Qb$ beyond its support. The concept of a robust standard quadratic optimization problem was introduced in~\cite{bomze2021trust} and formulated as a maximin problem. In that work, the authors investigated various uncertainty sets for the uncertain matrix and showed that the copositive relaxation gap coincides with the minimax gap. The authors also demonstrated that the robust StQP reduces to a deterministic StQP. 

\noindent {A setting in which the uncertain data matrix $\widetilde \Qb$ 
 contains a deterministic principal submatrix while the remaining entries are randomly distributed according to a known probability distribution was investigated in~\cite{bomze2022two}. In that model, the decision-maker makes a first decision before the uncertainty is revealed and a second decision afterward. The decision variable is therefore split into first- and second-stage components, and the resulting StQP is formulated as a two-stage stochastic optimization problem that does not admit a deterministic StQP counterpart.}       

\noindent The concept of a chance-constrained StQP was introduced in~\cite{bomze2025uncertain}, where it was formulated as the minimization of the Value-at-Risk of the random\footnote{Random parameters are denoted with a tilde sign.} variable $\x^{\top}\widetilde\Qb\x$. Under suitable distributional assumptions, the authors showed that this model is equivalent to a deterministic StQP.

\noindent \sloppy{The uncertain StQP has been tackled from a distributionally robust optimization (DRO) perspective in~\cite{HengArxiv}. There, the authors introduced the concept of a distributionally robust standard quadratic optimization problem, and established the reduction of the Wasserstein-DRO formulation to a robust optimization problem. Moreover, they unified three uncertainty models for the standard quadratic optimization problem—robust StQP, chance-constrained StQP, and distributionally robust StQP—and proved their equivalence under specific distributional assumptions.}

\noindent The present paper  tackles the StQP with uncertain data matrix $\widetilde \Qb$ via distributionally robust chance-constrained optimization with Wasserstein ambiguity. This optimization problem (DRCCP) is the distributionally robust counterpart of the Value-at-Risk minimization underlying the chance-constrained StQP of~\cite{bomze2025uncertain}: minimizing the smallest threshold $t$ satisfying 
$$\operatorname{Probability}[\x^{\top}\widetilde\Qb\x \le t]\ge 1-\varepsilon$$ 
is precisely the minimization of the Value-at-Risk of $\x^{\top}\widetilde\Qb\x$ at level $1-\varepsilon$. We require this robustly over all probability distributions within a Wasserstein distance of the empirical distribution of the sampled matrices; that is, we minimize the \emph{worst-case} Value-at-Risk. This distributionally robust chance-constrained StQP (DRCCStQP) is formalized in Section~\ref{sec:3.DRCCStQP}, and its worst-case Value-at-Risk interpretation is established in Corollary~\ref{cor:DRCCStQP-is-supVaR}.

\noindent In a distributionally robust chance-constrained optimization problem, the ``true'' probability measure generating the random vector is unknown. The chance constraint must be satisfied for all probability measures in an ambiguity set. Different ambiguity sets have been explored, for instance, moment-based~\cite{calafiore2006distributionally, ghaoui2003worst} and $\phi$-divergence~\cite{jiang2016data} ambiguity sets. We will focus on Wasserstein ambiguity sets. The 1-Wasserstein distance was first incorporated into a distributionally robust optimization problem in~\cite{pflug2007ambiguity} and was further developed in~\cite{wozabal2012framework,pflug2014multistage}. Later,
\cite{mohajerin2018data,zhao2018data} showed that for distributionally robust optimization problems with a 1-Wasserstein ball centered at the empirical distribution, the worst-case distribution can be explicitly constructed. Strong duality results for distance-based cost functions on $p$-Wasserstein balls with $p\geq 1$ or general lower semi-continuous cost functions were established in~\cite{gao2023distributionally,blanchet2019quantifying}. The MIP reformulation of DRCCP with Wasserstein ambiguity was shown for linear safety sets in~\cite{xie2021distributionally,chen2024data} and strengthened in~\cite{ho2022distributionally}. The right-hand side uncertainty setting of~\cite{ho2022distributionally} was extended to left-hand side uncertainty---closer to ours, since the uncertain data matrix enters the safety set coefficients rather than its right-hand side---with strong mixed-integer formulations in~\cite{honguyen2023strong}. For a broader survey of chance-constrained optimization under limited distributional information, covering both sampling-based and distributionally robust reformulations, see~\cite{kucukyavuz2022chance}.

\subsection{Notation, contributions and organization of the paper}

\subsubsection*{Norms and sets.}

\noindent An arbitrary norm on $\mathbb R^m$ is denoted by $\norm{\cdot}$ and its dual norm is denoted by $\norm{\cdot}^*$. The $\ell_q$-norms for $1\leq q \leq \infty$ are denoted by $\norm{\cdot}_q$. Later in the paper, we will define a  weighted $\ell_{\infty}$-norm with weight matrix $\Sb$, denoted by $\norm{\cdot}_{\infty,\Sb}$. The positive part of a number is denoted by $(\cdot)_+:=\max\{\cdot,0\}$ and the indicator function is denoted by $\mathbbm 1$. For a set $Z \subseteq \mathbb R^m$ we denote its complement by $Z^{\mathbf c}:=\mathbb R^m \setminus Z$. We define $\inf Z := \inf\{\z:\z\in Z\}$ and we use the convention that $\inf \emptyset = +\infty$. The special set $S(\x) \subseteq \mathbb R^m$, called the \emph{safety set}, depends on a vector $\x \in \mathbb R^n$; a point $\boldsymbol{\xi}\in S(\x)$ is called \emph{safe}, otherwise \emph{unsafe}. Given a distance function $d$ on $\mathbb R^m$, the distance of a point $\y \in \mathbb R^m$ to a set $Z \subseteq \mathbb R^m$ is defined as $\operatorname{dist}(\y,Z):=\inf_ {\z \in Z}d(\y,\z)$. If the distance is induced by a norm $d(\y,\z) = \norm{\y-\z}$, and $Z:=\{\z\in \mathbb R^m:\aa^{\top}\z \geq b\}$ is a closed halfspace of $\mathbb R^m$ where $\boldsymbol{0}\neq \aa \in \mathbb R^m$ and $b \in \mathbb R$, then the distance of $\y$ to $Z$ can be computed with the help of the dual norm ${\norm{\cdot}}^*$ as $\operatorname{dist}(\y, Z) = (b-\aa^{\top}\y)_+/\norm{\aa}^*$.

\subsubsection*{Probability measures and Wasserstein distance.}

\noindent Let $\mathcal{P}(\mathbb{R}^m)$ denote the set of all Borel probability measures on $\mathbb{R}^m$. For a point $\boldsymbol{\xi} \in \mathbb{R}^m$, the Dirac measure located at $\boldsymbol{\xi}$ is denoted by $\delta_{\boldsymbol{\xi}}$. Let $N \in \mathbb{N}$ be the sample size and put $[N] := \{1,\dots,N\}$. Given a sample $\{\boldsymbol{\widehat{\xi}}_i\}_{i \in [N]} \subset \mathbb{R}^m$, the empirical distribution is denoted by $\widehat{\mathbb P}_N := \frac{1}{N}\sum_{i \in [N]} \delta_{\boldsymbol{\widehat{\xi}}_i}$. For $p\in [1,\infty)$ the $p$-Wasserstein distance between two probability distributions  $\{\mathbb P, \mathbb P^{\prime} \}\subset \mathcal{P}(\mathbb{R}^m)$ is defined by
\begin{equation}\label{eqn:Wasserstein-distance}
W_p(\PP, \PP^{\prime}):= \left(\underset{\pi \in \Pi(\PP, \PP^{\prime})}{\inf}\displaystyle\int_{\mathbb R^m\times \mathbb R^m} \,  \norm{\boldsymbol{\xi}-\boldsymbol{\xi^{\prime}}}^p\, \text{d}\pi (\boldsymbol{\xi},\boldsymbol{\xi^{\prime}})\right)^{1/p},   
\end{equation}
where $\Pi(\PP, \PP^{\prime})$ denotes the set of all joint probability distributions $\pi$ on $\mathbb R^m \times \mathbb R^m$ with marginals $\PP$ and $\PP^{\prime}$. The Wasserstein ambiguity set centered at $\mathbb P^{\prime}\in \mathcal P(\mathbb R^m)$ with radius $\theta >0$ is denoted by $\mathbb B_{\theta,p} (\mathbb P^{\prime}):= \{\PP\in \mathcal P(\mathbb R^m) : W_p(\PP,\mathbb P^{\prime})\leq \theta \}$.

\subsubsection*{Probability distributions and random matrices.}

\noindent We denote by $\mathcal N(\mu, \sigma^2)$ and $\mathcal N_m(\boldsymbol{\mu}, \boldsymbol{\Sigma})$ the univariate normal distribution with mean $\mu$ and variance $\sigma^2$, and the multivariate $m$-dimensional normal distribution with mean $\boldsymbol{\mu}$ and covariance matrix $\boldsymbol{\Sigma}$, respectively. The Gaussian Orthogonal Ensemble of order $n$ is denoted by $\operatorname{GOE}(n)$ and the Wishart Ensemble of order $n$ with covariance matrix $\boldsymbol{\Sigma}$ and $k$ degrees of freedom is denoted by $\mathcal W_n(\boldsymbol{\Sigma},k)$. 
\noindent The cumulative distribution function of the standard normal distribution is denoted by $\Phi$, and its inverse is denoted by $\Phi^{-1}$. For a gamma-distributed random variable with shape parameter $a$ and scale parameter $b$, we denote its cumulative distribution function by $F(t;a,b) = P(a,t/b)$, where $P(a,z) = \gamma(a,z)/\Gamma(a)$. Here, $\gamma(\cdot,\cdot)$ denotes the lower incomplete gamma function and $\Gamma(\cdot)$ denotes the gamma function. For $\alpha \in (0,1)$, the inverse cumulative distribution function of a gamma-distributed random variable is denoted by $F^{-1}(\alpha;a,b) = P^{-1}(a,\alpha)b$, where $P^{-1}$ denotes the inverse of the regularized lower incomplete gamma function with respect to its second argument.

\subsubsection*{Matrix notation.}

\noindent We denote by $\mathcal S^n, \mathcal S^n_+$ and $\mathcal S^n_{++}$ the set of symmetric, positive semidefinite and positive definite matrices of order $n$. For a real matrix $\Ab = (a_{ij}) \in \mathbb R^{n\times n}$ we denote the trace operator by $\operatorname{tr}(\Ab)$. The entrywise maximum norm and the entrywise $1$-norm are denoted by $\norm{\Ab}_{\mathrm{max}} := \max_{i,j} |a_{ij}|$ and \sloppy $\|\Ab\|_{1,1} := \sum_{i,j} |a_{ij}|$, respectively. For a symmetric matrix $\Ab\in \mathcal S^{n}$ we denote the smallest and largest eigenvalue by $\lambda_{\mathrm{min}}(\Ab)$ and $\lambda_{\mathrm{max}}(\Ab)$, respectively; the symmetric vectorization of $\Ab \in \mathcal S^n$ is denoted by $\operatorname{svec}(\Ab)$. The symmetric vectorization operator $\operatorname{svec}: \mathcal S^n \to \mathbb R^m$ collects all entries on and above the diagonal of symmetric matrices in $\mathcal S^n$ into vectors in $\mathbb R^m$ with $m = \tfrac{n(n+1)}{2}$. For the ease of presentation, we will assume that for every symmetric matrix $\Ab = (a_{ij}) \in  \mathcal S^n$, the entries of its symmetric vectorization $\operatorname{svec}(\Ab)$ are arranged in the following way
$\operatorname{svec}(\Ab) = (a_{11},a_{22},\dots,a_{nn},\sqrt{2}a_{12},\sqrt{2}a_{13},\dots,\sqrt{2}a_{n-1,n})^{\top}$, i.e., the first $n$ entries of $\operatorname{svec}(\Ab)$ correspond to the diagonal entries of $\Ab$ and the last \sloppy $\frac{n(n+1)}{2} - n = \frac{n(n-1)}{2}$ entries of $\operatorname{svec}(\Ab)$ correspond to the off-diagonal entries of $\Ab$ multiplied with the factor $\sqrt{2}$. For any $\Ab, \Bb \in \mathcal S^{n}$ we denote the Frobenius inner product by $\langle \Ab , \Bb \rangle_F:= \operatorname{tr}(\Ab^{\top}\Bb)$, and the Frobenius norm by $\|\Ab\|_F = \sqrt{\langle \Ab , \Ab \rangle_F}$. Since off-diagonal entries are multiplied by $\sqrt 2$, this mapping is an isometry in the sense that for symmetric matrices $\Ab, \Bb \in \mathcal S^n$, $\langle \Ab , \Bb \rangle_F= \svec(\Ab)^{\top}\svec(\Bb)$, where the latter denotes the standard inner product on $\mathbb R^m$. In particular, this mapping is a linear bijection. The identity matrix, the matrix of all ones and the zero matrix of appropriate orders are denoted by $\Ib$, $\Eb = \e\e^{\top}$ and $\Ob$, respectively. If the order needs to be made explicit, we indicate it, e.g., as $\Ib_n$.

\subsubsection*{Main contributions and organisation of this paper}

\noindent This paper
\begin{itemize}
\item introduces the distributionally robust chance-constrained standard quadratic optimization problem and establishes a reduction of  the Wasserstein-DRCCP formulation to a mixed-integer optimization problem;
\item derives tractable deterministic reformulations for the DRCCStQP with suitable matrix norms;
\item {casts the DRCCStQP as worst-case Value-at-Risk minimization---an equivalence that follows directly from its epigraph form---which places it within, and unifies it with, existing frameworks for uncertain StQPs, in particular the chance-constrained StQP;}  
\item investigates different parameter combinations, explains the relationships between them, the optimal value, and running time.
\end{itemize}

\noindent The article is organized as follows. In Section~\ref{sec:2.Basics}, we review the distributionally robust chance-constrained optimization problem (DRCCP) and discuss well-known results for the Wasserstein ambiguity set. We present the structure of the worst-case distribution and the mixed-integer nonlinear programming (MINLP) reformulation of the DRCCP. Section~\ref{sec:3.DRCCStQP} defines the distributionally robust chance-constrained standard quadratic optimization problem (DRCCStQP). {Building on the worst-case Value-at-Risk objective set up in Section~\ref{sec:2.Basics}, we present the DRCCStQP as its instance for the loss $\x^{\top}\widetilde\Qb\x$.} Although the DRCCStQP is nonconvex in the decision variable, it is linear in the uncertain matrix, which enables a tractable robust reformulation. We apply results from Section~\ref{sec:2.Basics} and investigate suitable norms that make the resulting MINLP tractable. We choose the entrywise 1-norm and discuss constructing the worst-case distribution after solving the MINLP. We also present popular distributional models for random symmetric matrices. Those models (Wishart Ensemble and Gaussian Orthogonal Ensemble) will serve as the ``ground truth'' for the numerical experiments in the next section. In Section~\ref{sec:4.NumericalExperiments}, we run numerical experiments on positive definite instances (Wishart Ensemble) and indefinite instances (Gaussian Orthogonal Ensemble). We show how different model parameters affect the optimal value and the running time.  

%----------------------------------
%       2. Basics in DRCCP 
% ---------------------------------

\section{Distributionally robust chance-constrained optimization}\label{sec:2.Basics} 

\subsection{Basics}

\noindent In a chance-constrained optimization problem~\eqref{CCP}
\begin{equation}\label{CCP}\tag{CCP}
\begin{aligned}
\underset{\x \in \mathcal X}{\min} &&& \cc^{\top}\x\\
\mathrm{s.t.} &&& \mathbb P_{\mathrm{true}} [\boldsymbol{\tilde \xi} \notin S(\x)] \leq \varepsilon\,, 
\end{aligned}
\end{equation}
the goal is to minimize a deterministic linear function, where $\cc \in \mathbb R^n$ is a cost vector and $\x \in \mathcal X \subseteq \mathbb R^n$ is the decision variable from a feasible set $\mathcal X$, under chance constraints. The real-valued random vector $\boldsymbol{\tilde \xi} \in \mathbb R^m$ follows a known distribution $\mathbb P_{\mathrm{true}}$ and should fall outside a safety set $S(\x) \subseteq \mathbb R^m$ with probability at most $\varepsilon$, where $\varepsilon \in (0,1)$ is a given risk tolerance. When the distribution $\mathbb P_{\mathrm{true}}$ is unknown but there is access to an i.i.d. sample $\{\widehat{\boldsymbol{\xi}_i}\}_{i \in [N]}$ from $\mathbb P_{\mathrm{true}}$, the \eqref{CCP} can be approximated by the sample average approximation \eqref{SAA}

\begin{equation}\label{SAA}\tag{SAA}
\begin{aligned}
\underset{\x \in \mathcal X}{\min} &&& \cc^{\top}\x\\
\mathrm{s.t.} &&& \frac{1}{N}\sum_{i \in [N]} \mathbbm{1}_{\{\widehat{\boldsymbol{\xi}}_i\notin S(\x)\}} \leq \varepsilon\,, 
\end{aligned}
\end{equation}

\noindent where $\mathbbm{1}$ is an indicator function. This formulation replaces the distribution $\mathbb P_{\mathrm{true}}$ with the empirical distribution $\widehat{\mathbb P}_N$. The \eqref{SAA} is known to be very sensitive to the sample. As a way to robustify the empirical distribution, in a distributionally robust chance-constrained optimization problem
\begin{equation}\label{DRCCP}\tag{DRCCP}
\begin{aligned}
\underset{\x \in \mathcal X}{\min} &&& \cc^{\top}\x\\
\mathrm{s.t.} &&&
\sup_{\mathbb P \in \mathcal D} \mathbb P [\boldsymbol{\tilde\xi}\notin S(\x)]\leq \varepsilon\,,
\end{aligned}
\end{equation}
the chance constraint must be satisfied for all probability distributions contained in an ambiguity set $\mathcal D$ of distributions on $\mathbb R^m$ that also contains the empirical distribution. On this paper, we will focus on the choice $\mathcal D = \mathbb B_{\theta, 1}(\widehat{\mathbb P}_N)$, which consists of all probability measures within radius $\theta$ of the empirical measure based upon an i.i.d. sample $\{\boldsymbol{\widehat{\xi}}_i\}_{i \in [N]}$ from $\mathbb P_{\mathrm{true}}$, where the distance is given by the $1$-Wasserstein distance. 

\noindent For the rest of this section, we will review some useful and well-known results on \eqref{DRCCP} with Wasserstein ambiguity set $\mathcal D = \mathbb B_{\theta, 1}(\widehat{\mathbb P}_N)$. Let us recapitulate well-known results from~\cite{chen2024data}, which were already proven in~\cite{blanchet2019quantifying, gao2023distributionally}. The worst-case distribution of the uncertainty quantification problem with a Wasserstein ambiguity set centered at the empirical distribution supported on $N$ points is attained at a discrete distribution supported on at most $N+1$ points.

\begin{theorem}\label{thm:worst-case-distribution}\cite[Theorem~1]{chen2024data}. Let $p = 1$ and $S \subseteq \mathbb R^m$ be an open set. Assume that the sample $\{\widehat{\boldsymbol{\xi}}_i\}_{i \in [N]}$ is ordered in increasing distance to $S^{\mathbf c}$, i.e., 
$\operatorname{dist}(\widehat{\boldsymbol{\xi}}_i,S^{\mathbf c}) \leq  \operatorname{dist}(\widehat{\boldsymbol{\xi}}_j,S^{\mathbf c})$, for $1\leq i \leq j \leq N$ and put $i^* = \max\{i \in [N]: \operatorname{dist}(\widehat{\boldsymbol{\xi}}_i,S^{\mathbf c}) = 0\}$. Let $\boldsymbol{\xi}_i^* \in S^{\mathbf c}$ be a point closest to $\widehat{\boldsymbol{\xi}}_i$ in terms of the distance function $d$. Moreover, let $j^*:=\max\{j \in [N]\cup\{0\}: \sum_{i \in [j]}\operatorname{dist}(\widehat{\boldsymbol{\xi}_i}, S^{\mathbf c}) \leq \theta N\}$. Then, the uncertainty quantification problem 
\begin{equation*}
\sup_{\mathbb P \in \mathbb B_{\theta,1}(\widehat{\mathbb P}_N)} \mathbb P[\boldsymbol{\tilde\xi} \notin S] 
\end{equation*}
is attained for the following worst-case distribution $\mathbb P^* \in \mathbb B_{\theta,1}(\widehat{\mathbb P}_N)$.
\begin{enumerate}[label=(\roman*)]
    \item If $j^*=N$, then $\displaystyle\sup_{\mathbb P \in \mathbb B_{\theta,1}(\widehat{\mathbb P}_N)} \mathbb P[\boldsymbol{\tilde\xi} \notin S] = \mathbb P^*[\boldsymbol{\tilde\xi} \notin S] = 1$ for 
$$
 \mathbb P^* = \frac{1}{N}\sum_{i=1}^{i^*} \delta_{\widehat{\boldsymbol{\xi}}_i} + \frac{1}{N}\sum_{i=i^*+1}^N\delta_{\boldsymbol{\xi}_i^*}\,. 
$$
    \item If $j^*<N$, then $\displaystyle\sup_{\mathbb P \in \mathbb B_{\theta,1}(\widehat{\mathbb P}_N)} \mathbb P[\boldsymbol{\tilde\xi} \notin S] = \mathbb P^*[\boldsymbol{\tilde\xi} \notin S] = \frac{j^* + k^*}{N}$ for 
$$
\mathbb P^* = \frac{1}{N}\sum_{i=1}^{i^*}\delta_{\widehat{\boldsymbol{\xi}}_i}+ \frac{1}{N}\sum_{i=i^*+1}^{j^*}\delta_{\boldsymbol{\xi}_i^*}+ \frac{k^*}{N}\delta_{\boldsymbol{\xi}_{j^*+1}^*} + \frac{1-k^*}{N}\delta_{\widehat{\boldsymbol{\xi}}_{j^*+1}} +  \frac{1}{N}\sum_{i=j^*+2}^N\delta_{\widehat{\boldsymbol{\xi}}_i}\,,
$$
\noindent where $k^* = (\theta N - \sum_{i \in [j^*]}\operatorname{dist}(\widehat{\boldsymbol{\xi}_i},S^{\mathbf c}))/\operatorname{dist}(\widehat{\boldsymbol{\xi}}_{j^*+1},S^{\mathbf c})$.
\end{enumerate}
\end{theorem}

\noindent Problem~\eqref{DRCCP} with a Wasserstein ambiguity set can be reformulated as a deterministic optimization problem. 

\begin{theorem}\cite[Theorem~3]{chen2024data}. Let $p = 1$, then the distributionally robust chance-constrained problem \eqref{DRCCP} with Wasserstein ambiguity set $\mathcal D = \mathbb B_{\theta, 1}(\widehat{\mathbb P}_N)$ is equivalent to the following optimization problem:
\begin{equation*}
\begin{aligned}
\underset{\x, \gamma, \s}{\min} &&& \cc^{\top}\x&&&\\
\mathrm{s.t.} &&& \theta N + \e^{\top}\s \leq \varepsilon N \gamma&&&\\
&&& \operatorname{dist}(\widehat{\boldsymbol{\xi}}_i, S^{\mathbf c}(\x)) \geq \gamma - s_i&&& \forall i \in [N]\\
&&& \x \in \mathcal X, \gamma \in \mathbb R_+, \s \in \mathbb R_+^N\,.&&&
\end{aligned}
\end{equation*}
\end{theorem}

\noindent For the special case of a linear safety set given by an open halfspace, which we need later on, we obtain the following corollary, expressing the distance by a ratio involving the dual norm.

\begin{corollary} Let $p = 1$ and $S(\x) = \{\boldsymbol{\xi}\in \mathbb R^m: \aa(\x)^{\top}\boldsymbol{\xi} < b(\x)\}$ with $\boldsymbol{0}\neq\aa(\x)\in \mathbb R^m$ for all $\x \in \mathcal X$ and $b(\x)\in \mathbb R$. Then, the distributionally robust chance-constrained problem \eqref{DRCCP} with Wasserstein ambiguity set $\mathcal D = \mathbb B_{\theta, 1}(\widehat{\mathbb P}_N)$ is equivalent to 
\begin{equation*}
\begin{aligned}
\underset{\x, \gamma, \s}{\min} &&& \cc^{\top}\x&&&\\
\mathrm{s.t.} &&& \theta N + \e^{\top}\s \leq \varepsilon N \gamma&&&\\
&&&\frac{(b(\x)-\aa(\x)^{\top}\widehat{\boldsymbol{\xi}}_i)_+}{\norm{\aa(\x)}^*} \geq \gamma - s_i&&& \forall i \in [N]\\
&&& \x \in \mathcal X, \gamma \in \mathbb R_+, \s \in \mathbb R_+^N\,.&&&
\end{aligned}
\end{equation*}
\end{corollary}

\noindent We can eliminate the $(\cdot)_+$ operator by introducing binary variables and big-$M$ constraints. 

\begin{corollary}\cite[Proposition~2]{chen2024data}\label{cor:DRCCP-big-M} Let $p = 1$ and $S(\x) = \{\boldsymbol{\xi}\in \mathbb R^m: \aa(\x)^{\top}\boldsymbol{\xi} < b(\x)\}$ with $\aa(\x)\neq\boldsymbol{0}$ for all $\x \in \mathcal X$ and $b(\x)\in \mathbb R$. Then, the distributionally robust chance-constrained problem \eqref{DRCCP} with Wasserstein ambiguity set $\mathcal D = \mathbb B_{\theta, 1}(\widehat{\mathbb P}_N)$ is equivalent to 
\begin{equation*}
\begin{aligned}
\underset{\x, \gamma, \s, \u}{\min} &&& \cc^{\top}\x&&&\\
\mathrm{s.t.} &&& \theta N + \e^{\top}\s \leq \varepsilon N \gamma&&&\\
&&&\frac{b(\x)-\aa(\x)^{\top}\widehat{\boldsymbol{\xi}}_i}{\norm{\aa(\x)}^*} + Mu_i \geq \gamma - s_i&&& \forall i \in [N]\\
&&& M(1-u_i) \geq \gamma - s_i&&& \forall i \in [N]\\
&&& \x \in \mathcal X, \gamma \in \mathbb R_+, \s \in \mathbb R_+^N, \u \in \{0,1\}^N\,,&&&
\end{aligned}
\end{equation*}
where $M$ is a sufficiently large constant. 
\end{corollary}

\subsection{Chance constraints as worst-case Value-at-Risk objectives}\label{sec:2.wcVaR-role}

\noindent{In \eqref{CCP} and \eqref{DRCCP}, the chance constraint plays the role of a \emph{feasibility restriction}: it constrains the decision while a separate linear cost $\cc^{\top}\x$ is minimized. A chance constraint can equally serve as an \emph{objective}. Setting $\cc = \boldsymbol 0$ and lifting the threshold into the objective, the epigraphic problem $\inf\{t : \mathbb P[f(\x,\boldsymbol{\tilde\xi}) \le t] \ge 1-\varepsilon\}$ minimizes the $(1-\varepsilon)$-quantile of the loss $f(\x,\boldsymbol{\tilde\xi})$, that is, its Value-at-Risk; its distributionally robust version minimizes the worst-case Value-at-Risk over the ambiguity set $\mathcal D$. We record these two notions in general form here, so that the model of Section~\ref{sec:3.DRCCStQP} appears as the instance with loss $f(\x,\widetilde\Qb) = \x^{\top}\widetilde\Qb\x$, decision set $\mathcal X = \Delta$, and $\mathcal D = \mathbb B_{\theta,p}(\widehat{\mathbb P}_N)$.}

\begin{definition}
    Let $\tilde \zeta \in \mathbb R$ be a random variable with distribution $\mathbb P$ and $\varepsilon \in (0,1)$ be given, then the Value-at-Risk of $\tilde \zeta$  at level  $1-\varepsilon$ is defined as
$$
\operatorname{VaR}_{1-\varepsilon, \mathbb P}[\tilde\zeta] := \inf\{t\in \mathbb R: \mathbb P[\tilde\zeta\leq t] \geq 1-\varepsilon\}\,.
$$
If $f:\mathbb R^n \times \mathbb R^m \to \mathbb R$ is a function that depends on a decision variable $\x \in \mathcal X\subseteq \mathbb R^n$ and on a random vector $\boldsymbol{\tilde \xi}\in \mathbb R^m$, then the problem of minimizing the Value-at-Risk of the random variable $\tilde \zeta = f(\x,\boldsymbol{\tilde \xi})$ with distribution $\mathbb P$ at level $1-\varepsilon$ is defined as
$$
\begin{array}{lcl}
\underset{\x \in \mathcal X}{\inf}\operatorname{VaR}_{1-\varepsilon, \mathbb P}[f(\x,\boldsymbol{\tilde \xi})] =& \underset{(\x,t) \in \mathcal X \times \mathbb R}{\inf} & t\\
&\operatorname{s.t.} & \mathbb P[f(\x, \boldsymbol{\tilde \xi})\leq t] \geq 1-\varepsilon\,.
\end{array}
$$
\end{definition}

\begin{definition}
    Let $\tilde \zeta \in \mathbb R$ be a random variable whose distribution is unknown but assumed to lie in an ambiguity set $\mathcal D$ of probability distributions. The {\em worst-case Value-at-Risk}~\cite{ghaoui2003worst} of $\tilde \zeta$ at level $1-\varepsilon$ is given by
$$
\operatorname{wcVaR}_{1-\varepsilon, \mathcal D}[\tilde \zeta] := \inf\{t \in \mathbb R: \mathbb P[\tilde \zeta \leq t] \geq 1-\varepsilon \quad \forall \, \mathbb P \in \mathcal D\}
$$
and the problem of minimizing the worst-case Value-at-Risk of the random variable $\tilde\zeta = f(\x,\boldsymbol{\tilde \xi})$ is given by
\begin{equation*}
\begin{array}{lclc}
\underset{\x \in \mathcal X}{\inf} \operatorname{wcVaR}_{1-\varepsilon, \mathcal D}[f(\x, \boldsymbol{\tilde \xi})]:=& \underset{(\x,t) \in \mathcal X \times \mathbb R}{\inf} & t&\\
&\operatorname{s.t.} & \mathbb P[f(\x, \boldsymbol{\tilde \xi})\leq t] \geq 1-\varepsilon& \forall \, \mathbb P \in \mathcal D\,.
\end{array}
\end{equation*}
\end{definition}

\noindent {The following lemma characterizes the worst-case Value-at-Risk as}
$$
\operatorname{wcVaR}_{1-\varepsilon, \mathcal D}[\tilde \zeta] = \underset{\mathbb P \in \mathcal D}{\sup} \operatorname{VaR}_{1-\varepsilon, \mathbb P}[\tilde \zeta]\,.
$$

\begin{lemma}\label{lem:worst-case-is-sup}
    Let $\tilde \zeta \in \mathbb R$ be a random variable and let $\mathcal D \subset \mathcal P(\mathbb R)$ be an ambiguity set of probability measures. For any $\mathbb P \in \mathcal D$ and a given $\varepsilon\in (0,1)$ define the following sets:
$$
\mathfrak T_{\mathbb P}:=\{t\in \mathbb R: \mathbb P[\tilde \zeta \leq t]\geq 1-\varepsilon\}\,,\quad \mathfrak T:=\{t\in \mathbb R: \underset{\mathbb P \in \mathcal D}{\inf}\, \mathbb P[\tilde \zeta \leq t]\geq 1-\varepsilon\}\,.
$$
Then,
$$
\underset{\mathbb P\in \mathcal D}{\sup}\inf \mathfrak T_{\mathbb P} = \inf \mathfrak T\,.
$$
\end{lemma}
\begin{proof}
    Note that $\mathfrak T = \bigcap_{\mathbb P \in \mathcal D} \mathfrak T_{\mathbb P}$. Moreover, by construction $\mathfrak T$ and the $\mathfrak T_{\mathbb P}$ are intervals. In particular, $\mathfrak T = [\inf \mathfrak T,\infty)$ and $\mathfrak T_{\mathbb P} = [\inf\mathfrak T_{\mathbb P},\infty)$. The claim follows from
$$
[\inf \mathfrak T,\infty) = \mathfrak T = \bigcap_{\mathbb P \in \mathcal D} \mathfrak T_{\mathbb P} = \bigcap_{\mathbb P \in \mathcal D} [\inf\mathfrak T_{\mathbb P},\infty)= [\sup_{\mathbb P \in \mathcal D}\inf\mathfrak T_{\mathbb P},\infty)\,.
$$
\end{proof}

%----------------------------------
%       3. DRCCStQP
% ---------------------------------
\section{Distributionally robust chance-constrained StQP}\label{sec:3.DRCCStQP}

\subsection{Problem formulation and worst-case Value-at-Risk perspective}

\begin{definition}
    Let $\widetilde \Qb \in \mathcal S  ^n$ be a random symmetric matrix that follows an unknown distribution $\mathbb P_{\mathrm{true}}$ and let $\{\widehat{\Qb}_i\}_{i \in [N]}$ be a sequence of independent and identically distributed samples drawn according to $\mathbb P_{\mathrm{true}}$. Let $\varepsilon \in (0,1)$, $\theta >0$, and $p \in [1,\infty)$ be given. Consider the reference (empirical) distribution
    $\widehat\PP_N = \frac{1}{N}\sum_{i \in [N]}\delta_{\widehat{\Qb}_i}$
 generated by this sample, which in turn defines the ambiguity set $\mathbb B_{\theta,p} (\widehat\PP_N)$, the $p$-Wasserstein ball around $\widehat\PP_N$ with radius $\theta$ under the matrix norm~$\norm{\cdot}$. Then the Distributionally Robust chance-constrained Standard Quadratic Optimization Problem (DRCCStQP) is defined by the problem
\begin{equation}\label{DRCCStQP}\tag{DRCCStQP}
\begin{aligned}
\underset{(\x, t)\in \Delta \times \mathbb R}{\inf} &&& t\\
\mathrm{s.t.} &&& \inf_{\mathbb P \in \mathbb{B}_{\theta,p}(\widehat\PP_N)} \mathbb P [\x^{\top}\widetilde \Qb \x \leq t] \geq 1-\varepsilon\,.
\end{aligned}
\end{equation}
\end{definition}

\noindent{By construction, \eqref{DRCCStQP} is the instance of the worst-case Value-at-Risk objective introduced in Section~\ref{sec:2.wcVaR-role}, with loss $f(\x,\widetilde\Qb) = \x^{\top}\widetilde\Qb\x$, decision set $\mathcal X = \Delta$, and ambiguity set $\mathcal D = \mathbb B_{\theta,p}(\widehat{\mathbb P}_N)$. Specializing the characterization of Lemma~\ref{lem:worst-case-is-sup} to this instance yields the following equivalent form directly.}

\begin{corollary}\label{cor:DRCCStQP-is-supVaR} The \eqref{DRCCStQP} is equivalent to the following problem
$$
\inf_{\x\in \Delta} \sup_{\mathbb P\in \mathbb B_{\theta,p}(\widehat{\mathbb P}_N)} \operatorname{VaR}_{1-\varepsilon,\mathbb P}[\x^{\top}\widetilde\Qb\x]\,.
$$
\end{corollary}

\noindent {In words, \eqref{DRCCStQP} minimizes the worst-case Value-at-Risk of $\x^{\top}\widetilde\Qb\x$ at level $1-\varepsilon$ over all distributions in the Wasserstein ball $\mathcal D = \mathbb B_{\theta,p}(\widehat{\mathbb P}_N)$; equivalently, it robustifies, around the empirical measure, the Value-at-Risk minimization that defines the chance-constrained StQP of~\cite{bomze2025uncertain}.}

\subsection{\texorpdfstring{Equivalent reformulation via big-$M$ constraints}{Equivalent reformulation via big-M constraints}}

\noindent As a consequence of Corollary~\ref{cor:DRCCP-big-M}, we obtain the following theorem.

\begin{theorem}\label{thm:DRCCStQP-big-M-general-norm} Let $p=1$ and $\norm{\cdot}$ be any norm on $\mathbb R^m$, then the \eqref{DRCCStQP} is equivalent to the following optimization problem, a mixed-integer nonlinear program:
\begin{subequations}\label{DRCCStQP-general-norm}
\begin{align}
\underset{\x, t, \gamma, \s, \u}{\min} &&& t&&&\notag  \\
\mathrm{s.t.} &&& \theta N + \e^{\top}\s \leq \varepsilon N \gamma&&& \\
&&&\frac{t-\x^{\top}\widehat{\Qb}_i\x}{\norm{\operatorname{svec}(\x\x^{\top})}^*} + Mu_i \geq \gamma - s_i&&& \forall i \in [N] \label{constr:dual-norm}\\
&&& M(1-u_i) \geq \gamma - s_i&&& \forall i \in [N]\\
&&& \x \in \Delta, t\in \mathbb R, \gamma \in \mathbb R_+, \s \in \mathbb R_+^N, \u \in \{0,1\}^N\,,&&&
\end{align}
\end{subequations}
where $M$ is a sufficiently large constant. 
\end{theorem}
\begin{proof} Rename the decision variable $\x \gets (\x, t)$ and put 
$\boldsymbol{\tilde \xi} := \operatorname{svec}(\widetilde \Qb)$ and $\boldsymbol{\widehat \xi}_i := \operatorname{svec}(\widehat \Qb_i)$, for $i\in[N]$. Using the transformation 
$$
\begin{aligned}
    \x^{\top}\Qb\x = \operatorname{tr}(\x^{\top}\Qb\x) = \langle \x\x^{\top}, \Qb \rangle_F = \operatorname{svec}(\x\x^{\top})^{\top}\operatorname{svec}(\Qb)\,,
\end{aligned}
$$
we can see that
$$
\{\Qb \in \mathcal S^n:  \x^{\top}\Qb\x \leq t\} = \{\boldsymbol{\xi}\in \mathbb R^m: \operatorname{svec}(\x\x^{\top})^{\top}\boldsymbol{\xi} \leq t\}=:S(\x,t)\,.
$$
Observe that for any $\x \in \Delta$ we have $\operatorname{svec}(\x\x^{\top})\neq \boldsymbol{0}$. Moreover, the distance of a point to a halfspace is the same regardless of whether the halfspace is open or closed. The claim follows from Corollary~\ref{cor:DRCCP-big-M} after identification $\aa(\x,t) = \operatorname{svec}(\x\x^{\top})$ and $b(\x,t) = t$.
\end{proof}

\subsection{Simplification by use of entrywise maximum norm}

\noindent Constraint~\eqref{constr:dual-norm} above can generate challenges in optimization procedures. Indeed, observe that the Euclidean norm $\norm{\cdot}_2$ is self-dual and satisfies
$$
\norm{\operatorname{svec}(\x\x^{\top})}_2^* = \norm{\operatorname{svec}(\x\x^{\top})}_2 = \norm{\x\x^{\top}}_F = \x^{\top}\x\quad
\mbox{ for all } \x \in \Delta\, .$$
So for the choice $\norm{\cdot}=\norm{\cdot}_2$ we end up with a cubic constraint in \eqref{constr:dual-norm}, which could be hard to solve. {Rather, if we} were working with the vectorization operator instead of the symmetric vectorization operator, the $\ell_{\infty}$-norm would be perfect since it is the dual of the $\ell_1$-norm and it follows for every $\x \in \Delta$ that
$$
\norm{\operatorname{vec}(\x\x^{\top})}_{\infty}^* = \norm{\operatorname{vec}(\x\x^{\top})}_{1}  =  \sum_{i \in [N]} x_i^2 + 2\sum_{i< j}x_ix_j = \Big(\sum_{i \in [N]}x_i\Big)^2 = 1\,,
$$
which would render~\eqref{constr:dual-norm} a (possibly nonconvex) quadratic constraint.
Inspired by these observations we will construct a norm on $\mathbb R^m$ such that $\norm{\operatorname{svec}(\x\x^{\top})}^* = 1$ for every $\x \in \Delta$: a specifically weighted $\ell_{\infty}$-norm.

\begin{definition}
    For a positive definite matrix $\Sb\in \mathcal S^n_{++}$ we define the {\em weighted $\ell_{\infty}$-norm} on $\mathbb R^m$ as $\norm{\y}_{\infty,\Sb}:=\norm{\Sb\y}_{\infty}$.
\end{definition}

\noindent Positive definiteness of $\Sb$ implies that 
$\norm{\cdot}_{\infty,\Sb}$ is indeed a norm, as it satisfies positive definiteness, absolute homogeneity, and the triangle inequality. The choice $\Sb = \Ib$ yields the ``classical" $\ell_{\infty}$-norm on $\mathbb R^m$. Here, we will rather focus on the weight matrix 
\begin{equation}\label{weight matrix}
 \Sb = \left(\begin{array}{rc}
  \Ib_n   & \Ob \\\
   \Ob  & \frac{1}{\sqrt{2}}\,\Ib_{m-n}
\end{array}\right)\,, 
\end{equation}
where $m-n= \tfrac{n(n-1)}{2}$. This choice will render the entrywise maximum norm for matrices, when applied to the symmetric vectorization:
\begin{lemma}\label{lem:weighted-ell-inf-norm}
Consider the weighted $\ell_{\infty}$-norm on $\mathbb R^m$ with $m = \tfrac{n(n+1)}{2}$ and weight matrix \eqref{weight matrix}. Then, for any symmetric matrix $\Ab = (a_{ij})\in \mathcal S^n$, we have 
\begin{enumerate}[label=(\roman*)]
    \item $\|\operatorname{svec}(\Ab)\|_{\infty,\Sb} = \|\Ab\|_{\mathrm{max}} = \max_{i,j} |a_{ij}|$,
    \item $
\|\operatorname{svec}(\Ab)\|_{\infty,\Sb}^* = \|\Ab\|_{1,1}= \sum_{i, j}
|a_{ij}|$, and
\item $\norm{\operatorname{svec}(\x\x^{\top})}_{\infty,\Sb}^* = 1$, for every $\x \in \Delta$.
\end{enumerate}
\end{lemma}

\begin{proof}For the choice of $\Sb$ as in~\eqref{weight matrix}, the weighted $\ell_{\infty}$-norm induces the entrywise maximum norm of $\Ab$,
eliminating the factor $\sqrt{2}$ in front of the off-diagonal entries set by the symmetric vectorization operator, which yields $(i)$. To see $(ii)$, we observe that the dual norm of the entrywise maximum norm is given by the entrywise $1$-norm. Finally, $(iii)$ directly follows from $(ii)$ for $\Ab=\x\x\T$.
\end{proof}

\noindent As a result, use of this weighted $\ell_{\infty}$-norm $\|\cdot\|_{\infty,\Sb}$ simplifies problem~\eqref{DRCCStQP-general-norm}  significantly.

\begin{theorem}\label{DRCCStQP-maximum-norm}
    Let $p=1$ and $\norm{\cdot} = \norm{\cdot}_{\infty,\Sb}$ be the weighted $\ell_{\infty}$-norm from Lemma~\ref{lem:weighted-ell-inf-norm}, so that $\norm{\operatorname{svec}(\x\x^{\top})}_{\infty,\Sb}^* = 1$ for every $\x \in \Delta$. Then \eqref{DRCCStQP} is equivalent to the following MINLP:
\begin{subequations}\label{DRCCStQP_max_norm}
\begin{align}
\underset{\x, t, \gamma, \s, \u}{\min} &&& t&&& \notag \\
\mathrm{s.t.} &&& \theta N + \e^{\top}\s \leq \varepsilon N \gamma\label{constr:theta-epsilon_original}&&& \\
&&&t-\x^{\top}\widehat{\Qb}_i\x+ Mu_i \geq \gamma - s_i&&& \forall i \in [N]\label{constr:dual-norm=1}\\
&&& M(1-u_i) \geq \gamma - s_i&&& \forall i \in [N]\label{constr:big-M}\\
&&& \x \in \Delta, t\in \mathbb R, \gamma \in \mathbb R_+, \s \in \mathbb R_+^N, \u \in \{0,1\}^N\,,&&&
\end{align}
\end{subequations}
where $M$ is a sufficiently large constant. 
\end{theorem}
\begin{proof}
    Follows from Theorem~\ref{thm:DRCCStQP-big-M-general-norm} and Lemma~\ref{lem:weighted-ell-inf-norm}.
\end{proof}

\noindent If all the $\widehat{\Qb}_i$ are {positive-semidefinite}, problem \eqref{DRCCStQP_max_norm} becomes a mixed-integer convex optimization problem that solvers can solve directly.
Let us stress again that in view of matrix norms (and distances), we simply look at the entrywise maximum norm and the above construction $\|\operatorname{svec}(\Ab)\|_{\infty,\Sb}$ only helps with the symmetric vectorization.

\begin{remark} The choice of the entrywise maximum norm $\norm{\cdot}=\norm{\cdot}_{\mathrm{max}}$ simplifies problem \eqref{DRCCStQP} drastically, because then $\norm{\svec(\x\x^{\top})}^*$ vanishes completely (being equal to one). However, also using the Frobenius norm $\norm{\cdot}=\norm{\cdot}_F$, we can, in a similar way as in \cite[Proposition~1]{chen2024data}, arrive at a tractable formulation, replacing~\eqref{constr:theta-epsilon_original}  with a \emph{convex} quadratic constraint. This can be extended to the $\ell_q$-norm setting, using convex conical constraints.
\end{remark}

\subsection{\texorpdfstring{Preconditioning: data-driven bounds and improved choice of big-$M$}{Preconditioning: data-driven bounds and improved choice of big-M}}
\noindent The big-$M$ constant $M$ should be at least as large as $\tfrac{\theta}{\varepsilon}$, otherwise problem~\eqref{DRCCStQP_max_norm} becomes infeasible, as the following Lemma~\ref{lem:subproblem} demonstrates.

\begin{lemma}\label{lem:subproblem} Let $\varepsilon\in (0,\tfrac{1}{2}]$, $\theta >0$ and $M>0$ be fixed and consider the following feasibility subproblem of \eqref{DRCCStQP_max_norm}:
\begin{subequations}\label{DRCCStQP_max_norm_subproblem}
\begin{align}
\underset{\gamma, \s, \u}{\min} &&& 0&&& \notag \\
\mathrm{s.t.} &&& \theta N + \e^{\top}\s \leq \varepsilon N \gamma&&&\label{constr:subproblem_8a} \\
&&& M(1-u_i) \geq \gamma - s_i&&& \forall i \in [N]\label{constr:subproblem_8b}\\
&&&\gamma \in \mathbb R_+, \s \in \mathbb R_+^N, \u \in \{0,1\}^N\,.&&&\label{constr:subproblem_8c}
\end{align}
\end{subequations}
Then, \eqref{DRCCStQP_max_norm} has a feasible solution if and only if $\tfrac{\theta}{\varepsilon} \leq  M$.
\end{lemma}
\begin{proof} \underline{$\Longrightarrow$:} Suppose that \eqref{DRCCStQP_max_norm_subproblem} has a feasible solution $(\gamma^*, \s^*, \u^*)$ and denote by $k := \e^{\top}\u^*$ the number of indices with $u_i^*=1$. For every index $i\in[N]$ with $u_i^*=0$, \eqref{constr:subproblem_8b} implies $s_i^* \geq \gamma^* - M$, while every index with $u_i^*=1$ satisfies $s_i^*\ge\gamma^*$, since $M(1-u_i^*)=0$. Summing over all $i\in[N]$, we obtain the condition $\e^{\top}\s^* \geq (N-k)(\gamma^*-M)\,+\,k\gamma^* \;=\; N\gamma^* - (N-k)M$, which we can combine with \eqref{constr:subproblem_8a} to obtain $\theta N + N\gamma^* - (N-k)M \leq \varepsilon N\gamma^*$ or, after rearranging terms, $(N-k)M - \theta N \geq (1-\varepsilon)N\gamma^*$. Since $\s^* \geq \boldsymbol{0}$, constraint \eqref{constr:subproblem_8a} also yields $\gamma^* \geq \tfrac{\theta}{\varepsilon}$. As the coefficient $(1-\varepsilon)N$ is positive, substituting gives $(N-k)M - \theta N \geq (1-\varepsilon)N\tfrac{\theta}{\varepsilon}$, i.e., $(N-k)M \geq \tfrac{\theta N}{\varepsilon}$. In particular $N-k>0$, and therefore $\theta \leq \tfrac{N-k}{N}\varepsilon M \leq \varepsilon M$.\\ 
\noindent \underline{$\Longleftarrow$:}
Now suppose that $\theta \leq \varepsilon M$, then $(\gamma^*, \s^*, \u^*) = (M, \boldsymbol{0}, \boldsymbol{0})$ is a feasible solution for \eqref{DRCCStQP_max_norm_subproblem}.
\end{proof}

\noindent In fact, the big-$M$ constant should be larger than this $\tfrac{\theta}{\varepsilon}$. Now, we will derive a conservative approximation for $M$. First, compute 
$$
L:=\min_{i \in [N]}\min_{\x \in \Delta}\x^{\top}\widehat{\Qb}_i\x\,,\quad 
U:=\max_{i \in [N]}\max_{\x \in\Delta}\x^{\top}\widehat{\Qb}_i\x
$$
From constraint \eqref{constr:theta-epsilon_original} we obtain $\theta + (1-\varepsilon)\gamma \leq \gamma - s_j$ for some $j\in [N]$ for which it must also hold $u_j = 0$. Combining this with constraint \eqref{constr:dual-norm=1} we obtain 
$t \geq L + \theta + (1-\varepsilon)\gamma$ and in particular we obtain the following lower bound on $t^*$:
$$
t^*\geq L + \theta =: \underline{t}.
$$
\noindent The feasible point $\x^{(0)} = \tfrac{1}{n}\e$, $\gamma^{(0)} = \tfrac{\theta}{\varepsilon}$, $\s^{(0)} = \boldsymbol{0}$, $\u^{(0)} = \boldsymbol{0}$ yields the following upper bound on $t^*$:
$$
t^*\leq \max_{i \in [N]}(\tfrac{1}{n}\e)^{\top}\widehat{\Qb}_i(\tfrac{1}{n}\e) + \tfrac{\theta}{\varepsilon} =: \overline{t}\,.
$$
We obtain an upper bound on $\gamma^*$ using $\overline{t}$:
$$
\gamma^*\leq \frac{\overline{t}-L-\theta}{1-\varepsilon} =:\overline{\gamma}\,,
$$
which suggests using the following big-$M$:
\begin{equation}\label{eqn:big-M}
    M = \overline{\gamma} + (U - L - \theta)_+.
\end{equation}

\subsection{Construction of the worst-case distribution}\label{subsec:construction-wc-distr}

\noindent Now, we will sketch the construction of a worst-case distribution according to Theorem~\ref{thm:worst-case-distribution}. Let $(\x^*, t^*, \gamma^*, \s^*, \u^*)$ be an optimal solution to \eqref{DRCCStQP_max_norm}. This fixes the set $S(\x^*,t^*)$, which is now given by 
$$
S(\x^*,t^*) = \{\Qb \in \mathcal S^n: \x^{*\top}\Qb\x^* \leq t^*\}.
$$
First, compute $\operatorname{dist}(\widehat{\Qb}_i, S^{\mathbf c}(\x^*,t^*)) = (t^*-\x^{*\top}\widehat{\Qb}_i\x^*)_+$ for all $i\in [N]$ and rearrange the $\widehat{\Qb}_i$ in increasing distance to $S^{\mathbf c}(\x^*,t^*)$. Then, compute $i^*$, $j^*$ and $k^*$ according to Theorem~\ref{thm:worst-case-distribution}. Since $(\x^*, t^*, \gamma^*, \s^*, \u^*)$ is optimal, we are in case $j^*<N$. A point $\widehat{\Qb}_i$ is unsafe if and only if $u^*_i = 1$, in which case we also have $\gamma^* = s^*_i$. For a point $\widehat{\Qb}_i$, the closest unsafe point $\Qb^*_i \in S^{\mathbf c}(\x^*,t^*)$ (in the distance induced by the entrywise maximum norm $\norm{\cdot}_{\mathrm{max}}$) is given by
\begin{equation}\label{eqn:closest-unsafe-point}
\Qb^*_i := \widehat{\Qb}_i + (\gamma^*-s^*_i)\Eb\,.
\end{equation}
\noindent This is due to the following reason. If a point (matrix) is to be transported from $\x^{*\top}\widehat{\Qb}_i\x^*$ to the boundary $t^*$, then constraint \eqref{constr:dual-norm=1} must be tight at the optimum. Moreover, since for any $\x\in\Delta$ we have $\x^{\top}\Eb\x = 1$, we obtain
$$
t^* = \x^{*\top}\widehat{\Qb}_i\x^* + \gamma^*-s_i^* = \x^{*\top}\widehat{\Qb}_i\x^* + (\gamma^*-s_i^*) \x^{*\top}\Eb\x^*= \x^{*\top}(\widehat{\Qb}_i + (\gamma^*-s_i^*)\Eb)\x^*\,.
$$

\noindent Consider the partition $[N] = \mathcal I_{\mathrm{NT}} \, \cup \, \mathcal I_{\mathrm{T}}\, \cup\, \mathcal I_{\mathrm{MS}}$. The set $\mathcal I_{\mathrm{NT}}:=\mathcal I_{\mathrm{U}} \, \cup \, \mathcal I_{\mathrm{F}}$ denotes the indices $i\in [N]$ whose points will not be transported, either because they are already unsafe points ($\mathcal I_{\mathrm{U}}$), or because they are safe points that are too far away and transporting those points is extremely costly ($\mathcal I_{\mathrm{F}}$). The set $\mathcal I_{\mathrm{T}}$ consists of all $i\in [N]$ that are safe but closest to the unsafe set $S^{\mathbf c}(\x^*,t^*)$ and hence will be transported. Finally, $\mathcal I_{\mathrm{MS}}$ is a singleton corresponding to the point that will be mass split in case of nonintegral ${N}{\varepsilon}$, and empty otherwise.

\subsection{Distributional models for random symmetric matrices}

\noindent Before presenting distributional models for random symmetric matrices, consider the case of deterministic/noise-free samples. In this case, \eqref{DRCCStQP} reduces to an StQP.

\begin{corollary}\label{cor:noise-free} Let $p=1$ and $\norm{\cdot} = \norm{\cdot}_{\mathrm{max}}$. Let $\Qb_{\mathrm{nom}}\in \mathcal S^n$ be a nominal matrix and let $\mathbb P_{\mathrm{true}} = \delta_{\Qb_{\mathrm{nom}}}$. Suppose that the sample consists of $N$ times the nominal matrix $\{\widehat\Qb_i\}_{i\in[N]} =\{ \Qb_{\mathrm{nom}}\}_{i\in[N]}$, then \eqref{DRCCStQP} is equivalent to the following StQP:
$$
\min_{\x \in \Delta} \x^{\top}(\Qb_{\mathrm{nom}} + \tfrac{\theta}{\varepsilon} \Eb)\x\,.
$$    
\end{corollary}
\begin{proof} Symmetry argument combined a reversion of the epigraphic structure and a separation of the optimization problem in two summands.
\end{proof}

\noindent For the rest of this section and on preparation for Section~\ref{sec:4.NumericalExperiments} we will discuss two common distributional models for random symmetric matrices: the Wishart Ensemble and the Gaussian Orthogonal Ensemble. 

\begin{definition}\label{def:Wishart} 
    A random symmetric matrix $\widetilde\Wb \in \mathcal S^n$ is said to be distributed according to the Wishart Ensemble with covariance matrix $\sSigma$ and $k$ degrees of freedom (in symbols: $\widetilde \Wb \sim \mathcal W_n(\sSigma,k))$, if 
    $\widetilde\Wb  = \widetilde \Yb \widetilde \Yb^{\top}$ where $\widetilde \Yb = (\widetilde \y_1, \dots, \widetilde \y_k)$ and $\widetilde \y_1, \dots, \widetilde \y_k \sim \mathcal N_n(\boldsymbol{0},\sSigma)$ are i.i.d. random vectors from the $n$-variate normal distribution with mean $\boldsymbol{0}$ and positive definite covariance matrix $\sSigma$.
\end{definition}

\begin{definition}\label{def:GOE}
A random symmetric matrix $\widetilde\Gb = (\widetilde g_{ij}) \in \mathcal S^n$ is said to be distributed according to the Gaussian Orthogonal Ensemble (in symbols: $\widetilde\Gb \sim \operatorname{GOE}(n))$, if the diagonal entries $\widetilde g_{ii}$ are i.i.d. with $\widetilde g_{ii} \sim \mathcal N(0,2)$ for $1\leq i \leq n$ and the off-diagonal entries $\widetilde g_{ij}$ are i.i.d. with $\widetilde g_{ij} \sim \mathcal N(0,1)$ for $1\leq i < j\leq n$ and independent of the $\widetilde g_{ii}$.
\end{definition}

\noindent Now we define the two perturbation models used for our numerical experiments in Section~\ref{sec:4.NumericalExperiments}.

\begin{definition}\label{def:perturbation-models} Let $\Qb_{\mathrm{nom}}\in \mathcal S^n$ be an arbitrary nominal matrix and $\beta >0$ be a scaling parameter. In case the true distribution $\mathbb P_{\mathrm{true}}$ is such that
\begin{equation}\label{WishartE-perturbation-model}
\widetilde \Qb = \Qb_{\mathrm{nom}} + \beta \widetilde \Wb\, ,
\end{equation}
for $\widetilde \Wb \sim \mathcal W_n(\Ib,k)$ we speak of a Wishart Ensemble perturbation model
 and in case \begin{equation}\label{GOE-perturbation-model}
\widetilde \Qb = \Qb_{\mathrm{nom}} + \beta \widetilde \Gb\, ,
\end{equation}
for $\widetilde \Gb \sim \operatorname{GOE}(n)$, we speak of a GOE perturbation model.
\end{definition}

\noindent The Wishart Ensemble perturbation model can naturally produce positive definite instances (e.g., by setting $\Qb_{\mathrm{nom}} = \Ob$) and the GOE perturbation model can naturally produce indefinite instances (again, e.g., by setting $\Qb_{\mathrm{nom}} = \Ob$). Moreover, there exist reformulations for the problem of minimizing the Value-at-risk of the random variable $\x^{\top}\widetilde\Qb\x$ when $\widetilde \Qb$ is given in either of those forms:  

\begin{lemma}\cite[Proposition~4]{bomze2025uncertain}\label{lem:closed-form-VaR}
Let $\varepsilon \in (0,\tfrac{1}{2})$ and suppose that the distribution $\mathbb P_{\mathrm{true}}$ is either of the form \eqref{WishartE-perturbation-model} or \eqref{GOE-perturbation-model}, then  
$$
\operatorname{opt}:=\min_{\x\in \Delta} \operatorname{VaR}_{1-\varepsilon, \mathbb P_{\mathrm{true}}}[\x^{\top}\widetilde\Qb\x] = \min_{\x\in \Delta} \x^{\top}(\Qb_{\mathrm{nom}} + \tau \Ib)\x\,.
$$
In case of the Wishart Ensemble perturbation model \eqref{WishartE-perturbation-model} we have $\tau := 2\beta P^{-1}(\tfrac{k}{2},1-\varepsilon)$, where $ P^{-1}$ is the inverse of the regularized lower incomplete gamma function $P$ with respect to its second argument and in case of the GOE perturbation model \eqref{GOE-perturbation-model} we have $\tau := \sqrt{2}\beta \Phi^{-1}(1-\varepsilon)>0$, where $ \Phi^{-1}$ is the inverse cumulative distribution function of the standard normal distribution.
\end{lemma}

\section{Numerical experiments}\label{sec:4.NumericalExperiments}

\noindent In this section, we will investigate the performance of \eqref{DRCCStQP} for $p=1$ and $\norm{\cdot} = \norm{\cdot}_{\mathrm{max}}$ when the (unknown) true distribution is given as in Definition~\ref{def:perturbation-models}. We make a distinction: whether %without and with a 
the nominal matrix $\Qb_{\mathrm{nom}}$ is zero or not. All experiments were performed on a machine equipped with an Intel(R) Core(TM) i5-8350U CPU @ 1.70GHz (4 physical cores, 8 logical threads), 16 GB RAM, running Windows 11. We solved the optimization problems using \code{Gurobi} 11.0.2 on \code{Python} 3.12.0. Integer feasibility tolerance, relative MIP optimality gap, primal feasibility tolerance, and dual feasibility tolerance were all set to $10^{-5}$. We performed the experiments twice: first with \emph{data-independent} $M=5$ and then with \emph{data-dependent} $M$ according to \eqref{eqn:big-M}. There were no significant differences neither in optimal values nor in running times, even though the latter $M$ according to~\eqref{eqn:big-M} was much more conservative (for some examples we had $M\approx 250$). We report here the numerical experiments obtained for $M$ according to \eqref{eqn:big-M}.

\subsection{When the nominal matrix is zero}

\noindent First, we consider the case of a Wishart Ensemble perturbation model \eqref{WishartE-perturbation-model} with $k=n$ degrees of freedom. Then, we investigate the case of a GOE perturbation model \eqref{GOE-perturbation-model}. Both cases yield similar results, but the indefinite GOE instances take considerably longer to solve. 

\subsubsection{Positive definite instances}

\noindent Suppose that $\mathbb P_{\mathrm{true}}$ is of the form $
\widetilde \Qb = \beta \widetilde \Wb$, i.e., a Wishart Ensemble perturbation model \eqref{WishartE-perturbation-model} with $\Qb_{\mathrm{nom}} = \Ob$. By Lemma~\ref{lem:closed-form-VaR}, the problem of minimizing the Value-at-Risk of the random variable $\x^{\top}\widetilde\Qb\x$ at confidence level $1-\varepsilon$ has a closed form solution: $\operatorname{opt}= \tfrac{\tau}{n}$, where $\tau = 2\beta P^{-1}(\tfrac{n}{2},1-\varepsilon)$ and the minimizer is given by $\x^*=\tfrac{1}{n}\e$. 

\begin{example}[Baseline parameters]\label{ex:Qnom=0-wishart_baseline} For $n=10, \beta = 5$, and  $\varepsilon = 0.1$, we have that the optimal value of the minimization of the Value-at-Risk is equal to $\frac{\tau}{n} \approx 7.9956$. By Corollary~\ref{cor:DRCCStQP-is-supVaR}, \eqref{DRCCStQP} robustifies the Value-at-Risk around the empirical measure. If the radius $\theta$ is sufficiently large, we expect its optimal value to exceed $\tfrac{\tau}{n}$. We put $N = 20$ and generate a sample $\{\widehat \Wb_i\}_{i\in[N]} \sim \mathcal W_n(\Ib,n)$. We define $\widehat \Qb_i = \beta \widehat \Wb_i$, for $i\in [N]$ and compute $M\approx 151.42$ according to \eqref{eqn:big-M}. We take a small Wasserstein ball radius $\theta = 0.05$ and solve \eqref{DRCCStQP_max_norm}. In this case, we obtain $t^* \approx 6.8199$ (see Table~\ref{tab:Qnom=0-wishart_epsilon}, third row). In this case, there is one unsafe point, one point is transported, and no point is mass split. 
\end{example}

\begin{table}[H]
\centering
\scalebox{0.8}{\begin{tabular}{|ccccc|}
\hline
$\varepsilon$ & $\theta$ & $\beta$ & $n$ & $N$\\
\hline
0.1 & 0.05 & 5 & 10 & 20\\
\hline
\end{tabular}}
\caption{Baseline parameters}
\label{tab:baseline-parameters}
\end{table}

\noindent Now, we will analyze the effect of each parameter on the optimal value and running time.

\begin{example}[Structure of the solution for different values of $\varepsilon$]\label{ex:Qnom=0-wishart_epsilon-sweep} Now use the same setup as in Example~\ref{ex:Qnom=0-wishart_baseline} but also investigate the performance when changing the risk tolerance $\varepsilon \in \{0.01, 0.05, 0.1, 0.13, 0.2, 0.26\}$. Recall Subsection~\ref{subsec:construction-wc-distr} again for the notation on the set of unsafe points $\mathcal I_{\mathrm{U}}$, the set of transported points $\mathcal I_{\mathrm{T}}$, and the set of mass split point $\mathcal I_{\mathrm{MS}}$. In case $\varepsilon = 0.01$, since $N=20$ and $\varepsilon N = 0.2$, the solution possibly consists of mass splitting one point (Table~\ref{tab:Qnom=0-wishart_epsilon}, first row). Now, for $\varepsilon = 0.05$, the solution consists of transporting one point (Table~\ref{tab:Qnom=0-wishart_epsilon}, second row). For $\varepsilon = 0.2$, the transport plan transports two points and has two {unsafe} points (Table~\ref{tab:Qnom=0-wishart_epsilon}, fifth row). A point is also mass split for $\varepsilon \in \{0.13, 0.26\}$ (Table~\ref{tab:Qnom=0-wishart_epsilon}, fourth and sixth rows). For all rows, for $i \in \mathcal I_{\mathrm{U}}$ we have $u_i^* = 1$ and  $s_i^* = \gamma^*$, since for unsafe points, the constraint \eqref{constr:big-M} is tight. Independently of $\varepsilon$, the sorting of the $\widehat{\Qb}_i$ remains very stable and corresponds to the (decreasing) sorting of $(\tfrac{1}{n}\e)^{\top}\widehat{\Qb}_i(\tfrac{1}{n}\e)$. Both the optimal value of the minimization of the Value-at-Risk and the optimal value $t^*$ of \eqref{DRCCStQP_max_norm} decrease with increasing $\varepsilon$. As $\varepsilon$ increases, the running time also increases. Moreover, observe that in Table~\ref{tab:Qnom=0-wishart_epsilon}, $t^*$ is below $\operatorname{opt}$ for        $\varepsilon=0.1$ but above it for $\varepsilon=0.01$, as expected. The big-$M$ constant was independent of $\varepsilon$, staying at approximately $M\approx 150$. \textbf{\underline{Bottom line:} Due to its combinatorial nature, the running time increases with increasing $\varepsilon$.}
\end{example}

\begin{table}[H]
\centering
\scalebox{0.8}{\begin{tabular}{|ccccccc|}
\hline
$\varepsilon$ & opt & $t^*$ & $\mathcal I_{\mathrm{U}}$ & $\mathcal I_{\mathrm{T}}$ & $\mathcal I_{\mathrm{MS}}$ & running time\\
\hline
0.01 & 11.6046 & 12.0973 & $\{\}$ & $\{\}$ & $\{18\}$ & 0.29s\\
0.05 & 9.1535 & 8.0973 & $\{\}$ & $\{18\}$ & $\{\}$ & 0.28s\\
0.10 & 7.9936 & 6.8199 & $\{18\}$ & $\{3\}$ & $\{\}$ & 0.72s\\
0.13 & 7.5285 & 6.2122 & $\{18, 3\}$ & $\{\}$ & $\{19\}$ & 0.89s\\
0.20 & 6.7210 & 5.0455 & $\{18, 3\}$ & $\{19, 11\}$ & $\{\}$ & 1.27s\\
0.26 & 6.1934 & 4.7213 & $\{18, 3, 16\}$ & $\{19, 11\}$ & $\{6\}$ & 2.49s\\
\hline
\end{tabular}}
\caption{Wishart Ensemble perturbation model with $\Qb_{\mathrm{nom}} = \Ob$. Structure of the solution for different values of $\varepsilon$, for fixed $\theta = 0.05$, $\beta = 5$, $n=10$, $N=20$.}
\label{tab:Qnom=0-wishart_epsilon}
\end{table}

\noindent Now we will investigate the optimal value $t^*$ and running time for different parameter choices, using 20 random number generator seeds. 

\begin{example}[Changing the violation tolerance $\varepsilon$]\label{ex:Qnom=0-wishart_epsilon-seed-sweep} Same setup as in Example~\ref{ex:Qnom=0-wishart_epsilon-sweep} but for 20 different random number generator seeds. The violation tolerance $\varepsilon \in \{0.01,0.05, 0.1, 0.13, 0.2, 0.26\}$ varies, but we leave $\theta$, $\beta$, $n$, and $N$ constant as in Table~\ref{tab:baseline-parameters}. We compute the optimal value $t^*$ and the running time. The big-$M$ constant depended heavily on the seed but not much on $\varepsilon$ with $M\in[116,187]$. Figure~\ref{fig:Qnom=0-wishart-epsilon-seeds-sweep} shows the optimal value and running time on average, as well as their min-max range. As $\varepsilon$ increases, the optimal value $t^*$ decreases and the running time increases. \textbf{\underline{Bottom line:} Larger $\varepsilon$ allows for more optimistic decisions.}
\end{example}

  \begin{figure}[!htbp]
        \centering
\includegraphics[width=0.9\textwidth]{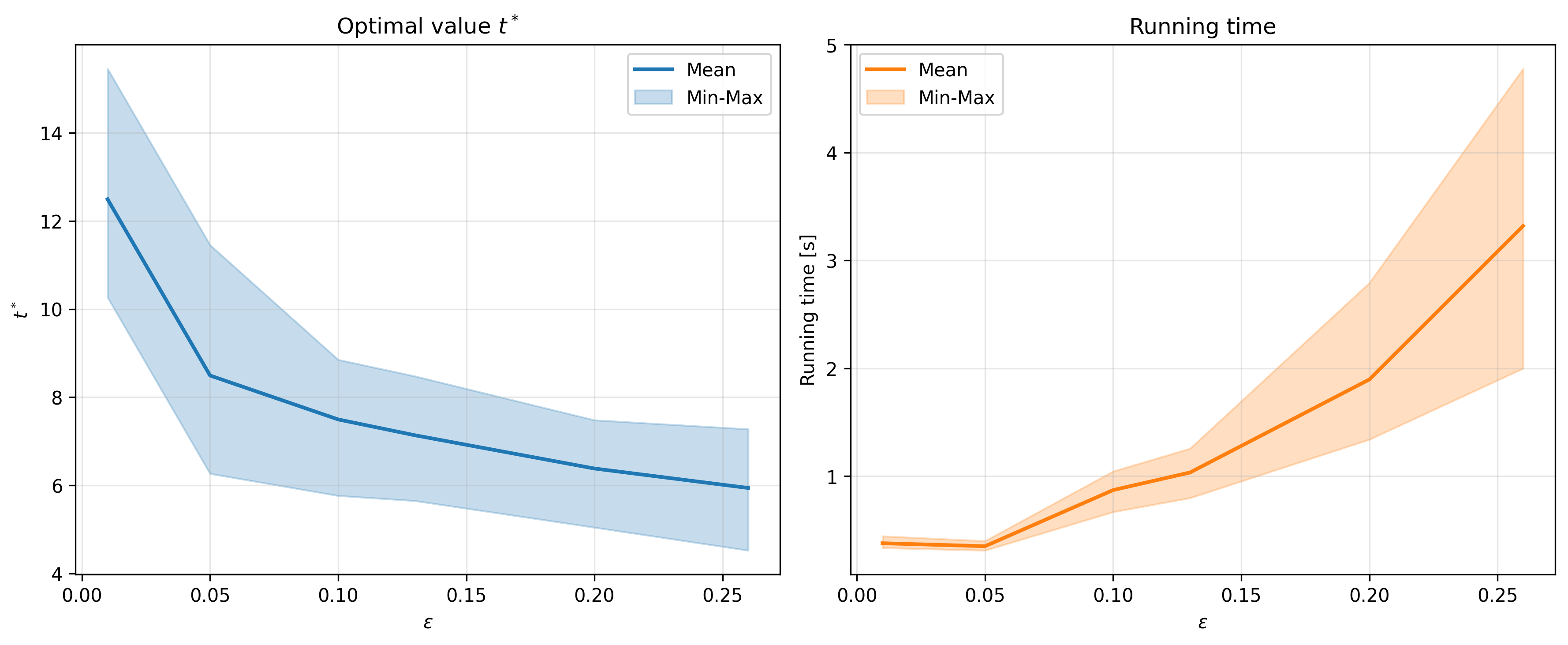}
\caption{Wishart Ensemble perturbation model with $\Qb_{\mathrm{nom}} = \Ob$. Optimal value $t^*$ and running time in seconds for different choices of $\varepsilon$, for fixed $\theta=0.05$, $\beta=5$, $n = 10$, $N=20$.}\label{fig:Qnom=0-wishart-epsilon-seeds-sweep}
    \end{figure}

\begin{example}[Changing the Wasserstein radius $\theta$]\label{ex:Qnom=0-wishart_theta-seed-sweep} Same setup as in Example~\ref{ex:Qnom=0-wishart_epsilon-sweep} but for 20 different random number generator seeds. The Wasserstein radius $\theta \in [10^{-3},10]$ varies, but we leave $\varepsilon$, $\beta$, $n$, and $N$ constant as in Table~\ref{tab:baseline-parameters}. The big-$M$ constant depended heavily on the seed but not much on $\theta$ with $M\in[115,250]$. We compute the optimal value decision $\x^*$, the optimal value $t^*$, and the running time. Figure~\ref{fig:Qnom=0-wishart-theta-seeds-sweep} shows the optimal value and running time on average, as well as their min-max range, in a logarithmic scale. As $\theta$ increases, the optimal value $t^*$ increases and the running time decreases. The optimal value increases drastically with increasing Wasserstein radius since more distributions are allowed. \textbf{\underline{Bottom line:} Larger $\theta$ buys robustness at the price of conservatism.}
 \end{example}

\begin{figure}[!htbp]
        \centering
\includegraphics[width=0.9\textwidth]{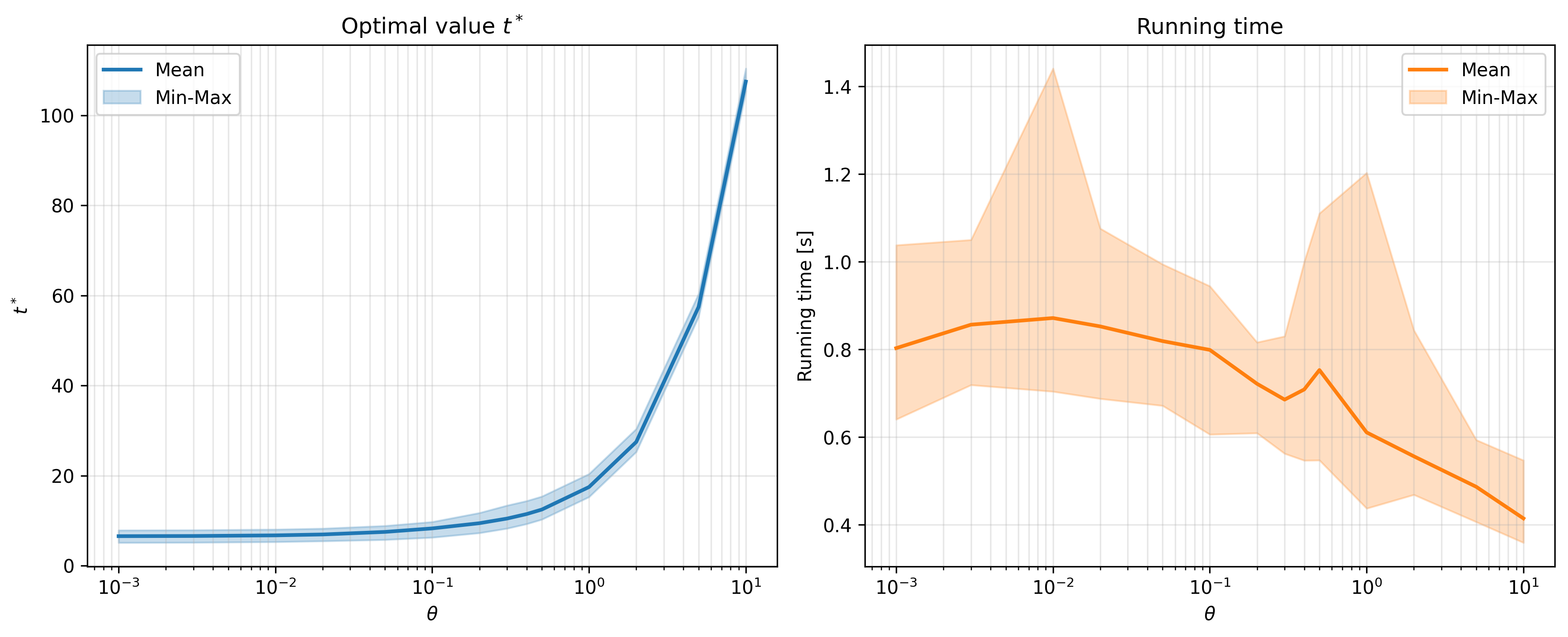}
\caption{Wishart Ensemble perturbation model with $\Qb_{\mathrm{nom}} = \Ob$. Optimal value $t^*$ and running time in seconds for different choices of $\theta$, for fixed $\varepsilon=0.1$, $\beta=5$, $n = 10$, $N=20$.}\label{fig:Qnom=0-wishart-theta-seeds-sweep}
    \end{figure}

\noindent Now we will look at the out-of-sample performance.

\begin{example}[Out-of-sample performance]\label{ex:Qnom=0-wishart_theta_out-of-sample} Continuing the discussion of Example~\ref{ex:Qnom=0-wishart_theta-seed-sweep}, we compute for every $\theta$ and every seed, the optimal values of the in-sample solutions $(\x^*,t^*)$ and the in-sample empirical probability
$$
\widehat{p}_{\mathrm{in}}:=\frac{1}{N}\sum_{i\in[N]}\mathbbm{1}_{\{\x^{*\top}\widehat{\Qb}_i\x^*\leq t^*\}}\,.
$$
For every $\theta$ and every seed, the in-sample empirical probability $
\widehat{p}_{\mathrm{in}}$ was at least 95\%, sometimes even 100\%, in any case greater than the prescribed $1-\varepsilon = 90\%$ by the chance constraint in \eqref{DRCCStQP}. Then, we construct another analogous, unseen, i.i.d. sample $\{\widehat \Qb_i^{\prime}\}_{i\in[N^{\prime}]}$, with $N^{\prime} = 10^5$ and evaluate the out-of-sample empirical probability
$$
\widehat{p}_{\mathrm{out}}:=\frac{1}{N^{\prime}}\sum_{i\in[N^{\prime}]}\mathbbm 1_{\{\x^{*\top}\widehat{\Qb}_i^{\prime}\x^*\leq t^*\}}\,.
$$
Figure~\ref{fig:Qnom=0-wishart-out-of-sample-theta} displays the out-of-sample empirical probability (left) and the conservatism gap (right), which we define as $\widehat{p}_{\mathrm{in}}- \widehat{p}_{\mathrm{out}}$. Small values of $\theta$ perform poorly on the unseen data, with an out-of-sample empirical probability of $\widehat{p}_{\mathrm{out}}\approx 70\%$, well below the prescribed $90\%$, implying poor generalization. Less conservative/higher values of $\theta$ perform better. For $\theta \approx 0.9$ we even have $\widehat{p}_{\mathrm{out}} \approx 90\%$ on average. The conservatism gap and the min-max bandwidth also reduce with increasing $\theta$. It is remarkable that for one seed, the out-of-sample empirical probability $\widehat{p}_{\mathrm{out}}$ even outperformed the in-sample empirical probability $\widehat{p}_{\mathrm{in}}$. The generalization is examined in greater detail in Table~\ref{tab:out_of_sample_theta}. There, we show, for $\theta \in \{0.1, \dots, 0.5\}$ the out-of-sample empirical probabilities (mean $\pm$ (standard deviation) and the confidence $1-\eta$, which we define as the fraction of the seeds where the out-of-sample empirical probability $\widehat{p}_{\mathrm{out}}$ was at least the prescribed confidence level $1-\varepsilon$:
$$
\mathbb{P}_{\mathrm{true}}^N\left[\widehat{\mathbb P}_N\left[\x^{*\top}\widetilde\Qb\x\leq t^*\right]\geq 1-\varepsilon\right]=:1-\eta.
$$
A Wasserstein radius $\theta \geq 0.3$ already yields an out-of-sample empirical probability higher than 90\% on average; see Table~\ref{tab:out_of_sample_theta}. \textbf{\underline{Bottom line:} In order to improve the out-of-sample performance, the Wasserstein radius $\theta$ should be as large as possible.}
\end{example}

 \begin{figure}[!htbp]
        \centering
\includegraphics[width=0.9\textwidth]{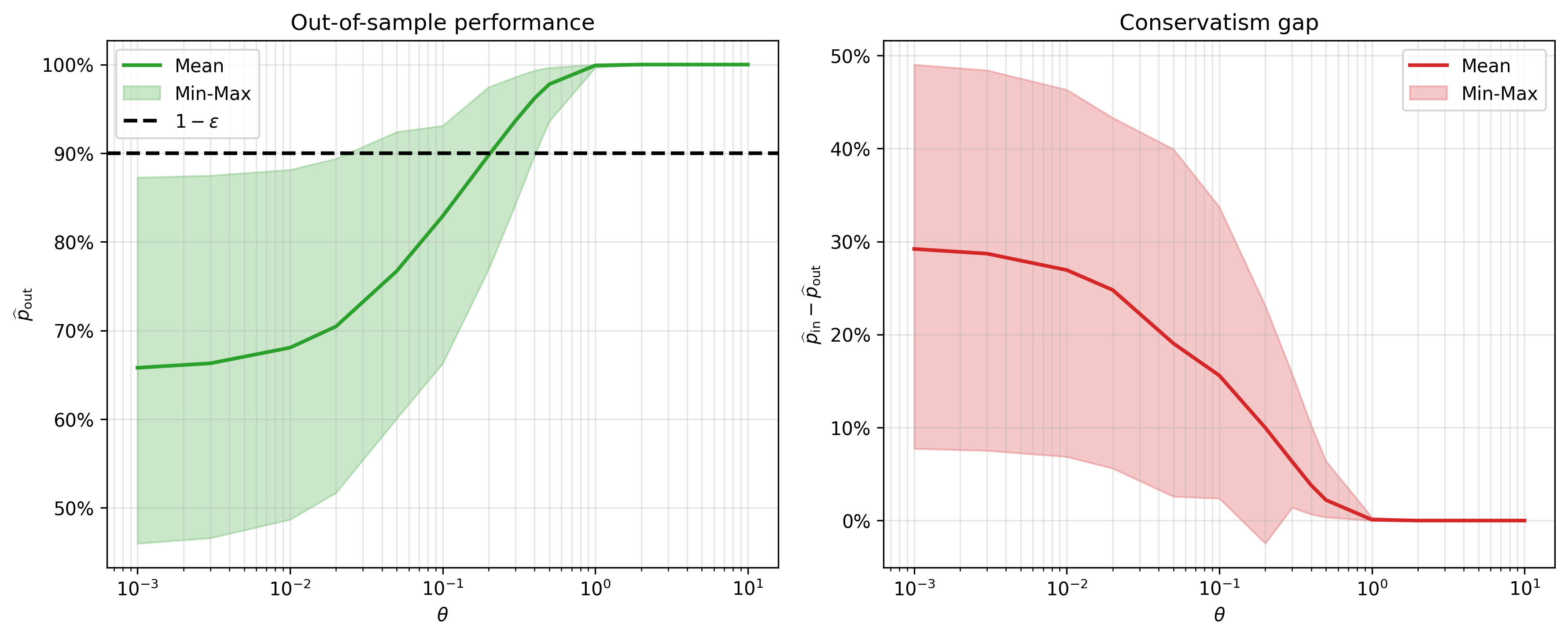}
\caption{Wishart Ensemble perturbation model with $\Qb_{\mathrm{nom}} = \Ob$. Out-of-sample performance and conservatism gap for different choices of $\theta$, for fixed $\varepsilon=0.1$, $\beta = 5$, $n=10$ and $N=20$.}\label{fig:Qnom=0-wishart-out-of-sample-theta}
    \end{figure}

\begin{table}[H]
\centering
\scalebox{0.8}{\begin{tabular}{|ccc|}
\hline
$\theta$ & $\widehat{p}_{\mathrm{out}}$ & $1-\eta$\\
\hline
0.1 & $82.9\%\pm(6.6\%)$ & 15\% \\
0.2 & $89.8\%\pm(5.2\%)$ & 45\% \\
0.3 & $93.7\%\pm(3.5\%)$ & 85\% \\
0.4 & $96.2\%\pm(2.3\%)$ & 95\% \\
0.5 & $97.8\%\pm(1.4\%)$ & 100\% \\
\hline
\end{tabular}}
\caption{Out-of-sample performance for different $\theta$ over 20 independent runs. The out-of-sample empirical probability $\widehat{p}_{\mathrm{out}}$ is given as mean $\pm$ (standard deviation). The confidence $1-\eta$ is the fraction of seeds where the out-of-sample empirical probability $\widehat{p}_{\mathrm{out}}$ was at least the prescribed confidence level $1-\varepsilon = 90\%$.}
\label{tab:out_of_sample_theta}
\end{table}

\begin{example}[Changing the scale $\beta$]\label{ex:Qnom=0-wishart_n-beta-sweep} This uses the same setup as in Example~\ref{ex:Qnom=0-wishart_epsilon-sweep} but for 20 different random number generator seeds. The scale $\beta \in [2, 12]$ varies, but we leave $\varepsilon$, $\theta$, $n$, and $N$ constant as in Table~\ref{tab:baseline-parameters}. The big-$M$ constant ranged from $M\approx 0.5$ for small values of $\beta$ to $M\approx 360$ for large values of $\beta$, across all seeds. We compute the optimal value $t^*$ and the running time. Figure~\ref{fig:Qnom=0-wishart-beta-seeds-sweep} shows the optimal value and running time on average, as well as their min-max range. As $\beta$ increases, the optimal value $t^*$ increases linearly while the running time stays approximately the same, because the samples are positive definite. \textbf{\underline{Bottom line:} Larger scale $\beta$ for positive definite instances implies larger quadratic forms, thus higher optimal values.}
\end{example}
  \begin{figure}[!htbp]
        \centering
\includegraphics[width=0.9\textwidth]{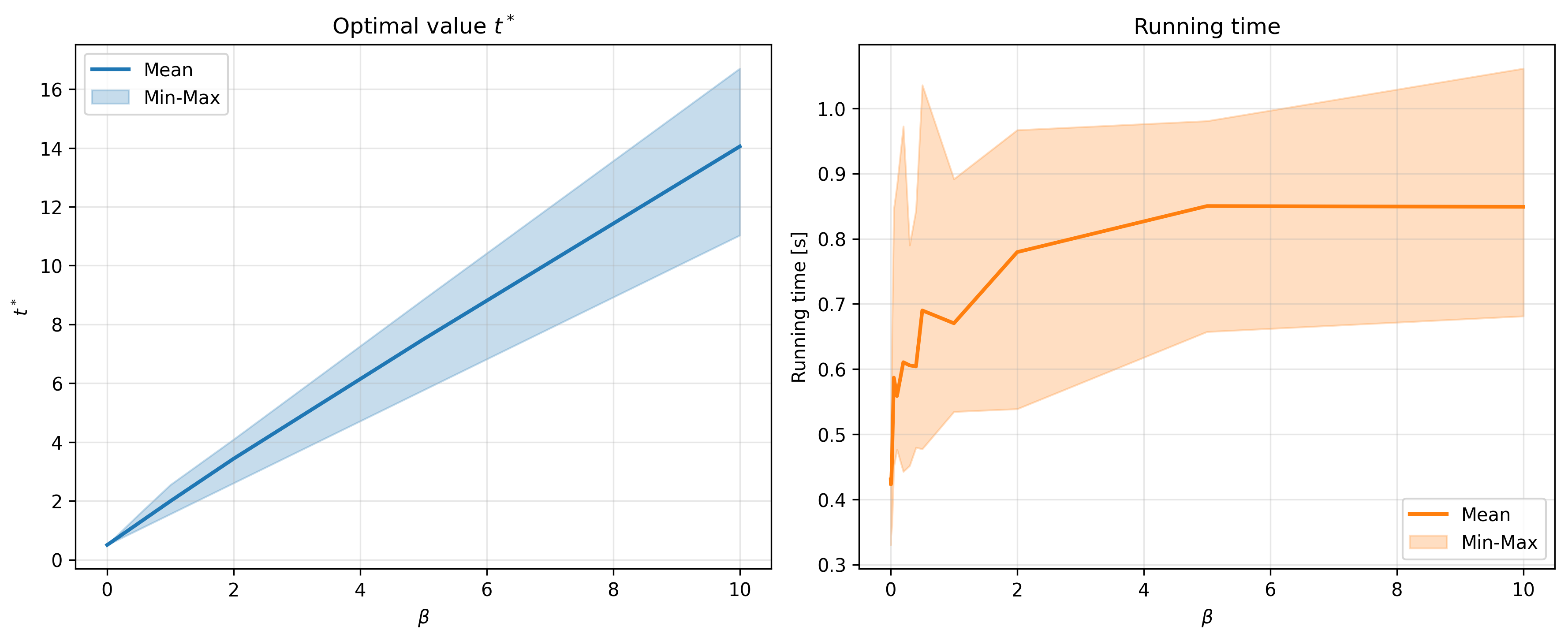}
\caption{Wishart Ensemble perturbation model with $\Qb_{\mathrm{nom}} = \Ob$. Optimal value $t^*$ and running time in seconds for different choices of $\beta$, for fixed $\varepsilon=0.1$, $\theta=0.05$, $n=10$, $N=20$.}\label{fig:Qnom=0-wishart-beta-seeds-sweep}
    \end{figure}

\begin{example}[Changing the dimension size $n$]\label{ex:Qnom=0-wishart_n-seed-sweep} Same setup as in Example~\ref{ex:Qnom=0-wishart_epsilon-sweep} but for 20 different random number generator seeds. The dimension size $n \in \{5, 10, 15, 20\}$ varies, but we leave $\varepsilon$, $\theta$, $\beta$, and $N$ constant as in Table~\ref{tab:baseline-parameters}. Naturally, the big-$M$ constant according to \eqref{eqn:big-M} was dependent on $n$ with $M$ ranging from $M\approx 80$ to $M\approx 230$. We compute the optimal value $t^*$ and the running time. Figure~\ref{fig:Qnom=0-wishart-n-seeds-sweep} displays the optimal value and the running time as boxplots. The figure shows the interquartile range, median, and mean. The median and mean of the optimal value $t^*$ decrease as $n$ increases, and so do the interquartile range and whisker lengths, indicating less variability. The running time shows the opposite behavior: the median, mean, interquartile range, and whisker length increase. \textbf{\underline{Bottom line:}  {Running time increases with the dimension $n$ and is still manageable for $n=20$, for positive definite instances}}.
 \end{example}
  \begin{figure}[!htbp]
        \centering
\includegraphics[width=0.9\textwidth]{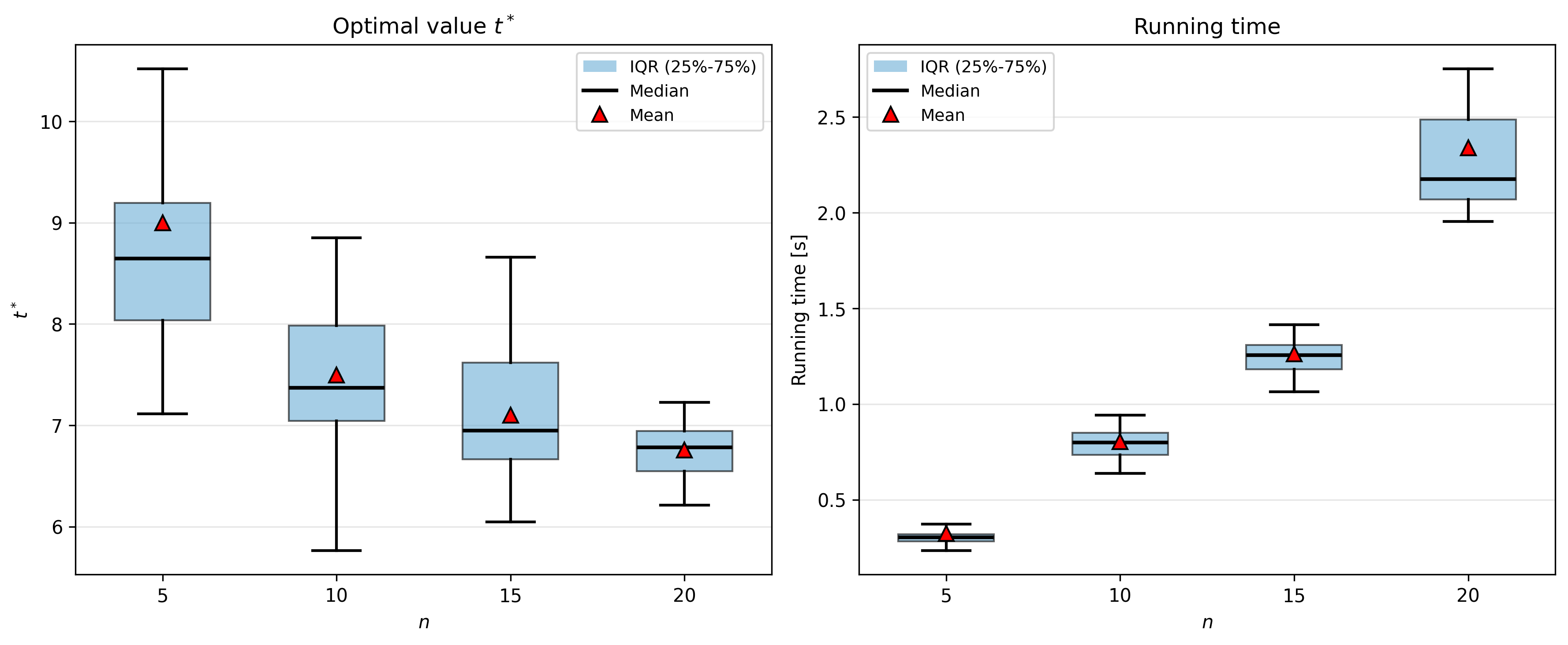}
\caption{Wishart Ensemble perturbation model with $\Qb_{\mathrm{nom}} = \Ob$. Optimal value $t^*$ and running time in seconds for different choices of $n$, for fixed $\varepsilon=0.1$, $\theta=0.05$, $\beta = 5$, $N=20$.}\label{fig:Qnom=0-wishart-n-seeds-sweep}
    \end{figure}

\begin{example}[Changing the sample size $N$]\label{ex:Qnom=0-wishart_N-seed-sweep} Same setup as in Example~\ref{ex:Qnom=0-wishart_epsilon-sweep} but for 20 different random number generator seeds. The sample size $N \in \{15, 20, 25, 30\}$ varies, but we leave $\varepsilon$, $\theta$, $\beta$, and $n$ constant as in Table~\ref{tab:baseline-parameters}. The big-$M$ constant was very seed dependent and did not vary  much with increasing $N$ staying at $M\approx 150$. We compute the optimal value $t^*$ and the running time. Figure~\ref{fig:Qnom=0-wishart-N-seeds-sweep} displays the optimal value and the running time as boxplots. The figure shows the interquartile range, median, and mean. The median and mean of the optimal value $t^*$ increase with increasing $N$. As expected, the running time increases with increasing sample size $N$, since an additional sample represents two additional constraints in the MINLP~\eqref{DRCCStQP_max_norm}. {\textbf{\underline{Bottom line:} Oversampling does not meaningfully change the optimal value and comes at the cost of increasing the running time.}}
\end{example}
  \begin{figure}[!htbp]
        \centering
\includegraphics[width=0.9\textwidth]{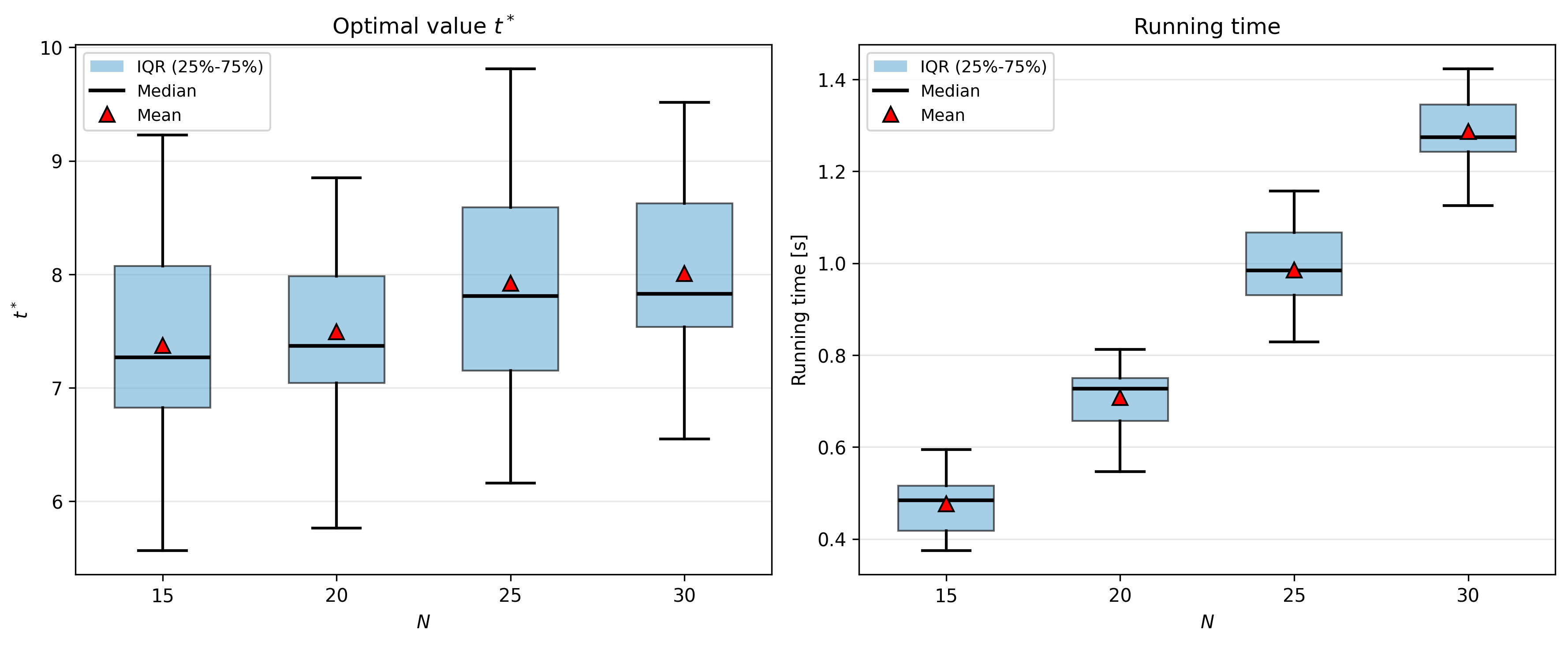}
\caption{Wishart Ensemble perturbation model with $\Qb_{\mathrm{nom}} = \Ob$. Optimal value $t^*$ and running time in seconds for different choices of $N$, for fixed $\varepsilon=0.1$, $\theta=0.05$, $\beta = 5$, $n=10$.}\label{fig:Qnom=0-wishart-N-seeds-sweep}
\end{figure}

\subsubsection{Indefinite instances}\label{subsec:indefinite-instances}
\noindent Now, we will investigate the case of the GOE perturbation model \eqref{GOE-perturbation-model}. Suppose that $\mathbb P_{\mathrm{true}}$ is of the form $
\widetilde \Qb = \beta \widetilde \Gb$, i.e., a GOE perturbation model \eqref{GOE-perturbation-model} with $\Qb_{\mathrm{nom}} = \Ob$. By Lemma~\ref{lem:closed-form-VaR}, the problem of minimizing the Value-at-Risk of the random variable $\x^{\top}\widetilde\Qb\x$ at confidence level $1-\varepsilon$ has a closed form solution: $\operatorname{opt}= \tfrac{\tau}{n}$, where $\tau = \sqrt{2}\beta \Phi^{-1}(1-\varepsilon)$ and the minimizer is given by $\x^*=\tfrac{1}{n}\e$. 

\begin{example}[Baseline parameters]\label{ex:GOE-baseline}  For $n=10, \beta = 5$, and  $\varepsilon = 0.1$, we have that the optimal value of the minimization of the Value-at-Risk is equal to $\frac{\tau}{n} \approx 0.9062$. By Corollary~\ref{cor:DRCCStQP-is-supVaR}, \eqref{DRCCStQP} is the robustification of the Value-at-Risk around the empirical measure. If the radius $\theta$ is sufficiently large, we expect its optimal value to exceed $\tfrac{\tau}{n}$. We put $N = 20$ and generate a sample $\{\widehat \Gb_i\}_{i\in[N]}\sim \operatorname{GOE}(n)$. We define $\widehat \Qb_i = \beta \widehat \Gb_i$, for $i\in [N]$, and compute a big-$M$ number according to \eqref{eqn:big-M} of $M = 53$. We take a small Wasserstein ball radius $\theta = 0.05$ and solve \eqref{DRCCStQP_max_norm}. In this case, we obtain $t^* \approx {0.8325}$ (see Table~\ref{tab:GOE-epsilon}, third row). In this case, there are no {unsafe points}, two points are transported, and no point is mass split.
\end{example}

\begin{example} [Structure of the solution for different values of $\varepsilon$]\label{ex:Qnom_indef_GOE_epsilon-sweep} For Table~\ref{tab:GOE-epsilon}, we use the same setup as in Example~\ref{ex:Qnom=0-wishart_epsilon-sweep} but with $\widehat \Qb_i =\beta \widehat \Gb_i$, for $i\in [N]$. The $\widehat \Qb_i$ are indefinite with $\lambda_{\mathrm{min}}(\widehat \Qb_i) \in [-38.03,-19.05]$ and $\lambda_{\mathrm{max}}(\widehat \Qb_i) \in [21.68, 33.16]$, for $i \in [N]$. Compared to Example~\ref{ex:Qnom=0-wishart_epsilon-sweep}, the noise-free optimal values and $t^*$ are lower; $t^*$ even becomes negative. 
The running times are considerably longer, because the matrices $\widehat \Gb_i$ are indefinite. Each problem \eqref{DRCCStQP_max_norm} has $N=20$ nonconvex quadratic constraints. For every $\varepsilon$, the big-$M$ constant according to \eqref{eqn:big-M} was around $M\approx 55$. The sorting of the $\widehat{\Qb}_i$ remains also stable and corresponds to the (decreasing) sorting of $(\tfrac{1}{n}\e)^{\top}\widehat{\Qb}_i(\tfrac{1}{n}\e)$. Moreover, many indices appear in the solutions of both Examples (e.g., $\{3, 13, 14, 15, 16, 18, 19\}$). This is because the matrices $\widehat{\Wb}_i$ and $\widehat{\Gb}_i$ are sampled using the same random number generator seed. Since, $k=n$, we sampled $\widehat{\Nb}_i$ entrywise from the standard normal distribution $(\widehat{\Nb}_i)_{k,\ell} \sim \mathcal N(0,1)$ for $k, \ell \in [n]$ and then put $\widehat \Wb_i = \widehat \Nb_i \widehat \Nb_i^{\top}$, and $\widehat \Gb_i = \tfrac{1}{\sqrt{2}}(\widehat \Nb_i + \widehat \Nb_i^{\top})$, for $i \in [N]$. \textbf{\underline{Bottom line:} Due to its nonconvex nature, for indefinite instances, running times increase substantially with increasing $\varepsilon$.}
\end{example}

\begin{table}[H]
\centering
\scalebox{0.8}{\begin{tabular}{|ccccccc|}
\hline
$\varepsilon$ & opt & $t^*$ & $\mathcal I_{\mathrm{U}}$ & $\mathcal I_{\mathrm{T}}$ & $\mathcal I_{\mathrm{MS}}$ & running time\\
\hline
0.01 & 1.6450 & 5.3325 & $\{\}$ & $\{\}$ & $\{16\}$ & 2.59s\\
0.05 & 1.1631 & 1.3325 & $\{\}$ & $\{16\}$ & $\{\}$ & 2.33s\\
0.10 & 0.9062 & 0.8325 & $\{\}$ & $\{13, 1\}$ & $\{\}$ & 9.96s\\
0.13 & 0.7965 & 0.5827 & $\{13\}$ & $\{7\}$ & $\{14\}$ & 11.66s\\
0.20 & 0.5951 & -0.1367 & $\{3, 1\}$ & $\{16, 15\}$ & $\{\}$ & 17.53s\\
0.26 & 0.4549 & -0.6727 & $\{3, 1, 16\}$ & $\{19, 15\}$ & $\{13\}$ & 39.54s\\
\hline
\end{tabular}}
\caption{GOE perturbation model with $\Qb_{\mathrm{nom}} = \Ob$. Structure of the solution for different values of $\varepsilon$, for fixed $n=10$, $N=20$, $\beta = 5$, $\theta = 0.05$.}
\label{tab:GOE-epsilon}
\end{table}

\begin{example} [Changing the scale $\beta$ and the dimension $n$]\label{ex:indef_changing_beta_and_n} 
\noindent Then, we performed the same numerical experiments as in the previous examples. The effects of varying the parameters on the GOE perturbation model remain the same as on the Wishart Ensemble perturbation model for almost all parameters, but for the scale $\beta$, as shown in Figure~\ref{fig:Qnom=0-goe-beta-seeds-sweep}. This is because the quadratic form $\beta\x^{\top}\widetilde
\Gb\x$ is a normal random variable with mean 0 and variance $2\beta^2\|\x\|_2^4$. Increasing $\beta$ does not change the mean, but it increases the variance, thus the dispersion shown in Figure~\ref{fig:Qnom=0-goe-beta-seeds-sweep}. The big-$M$ constant depended heavily on $\beta$ and ranged from $M\approx 0.5$ for smaller values of $\beta$ to $M\approx 150$ for larger values of $\beta$.

\noindent Notably, the running times explode as dimension $n$ increases. In Table~\ref{tab:n-Wishart-GOE-no-Qnom}, we compare the optimal values $t^*$ and running times for both models. The running time for GOE and $n=20$ ranged from 515s to 11500s, the computed big-$M$ was about $M=60$ for all iterations. \code{Gurobi} found the global optimum in 77 out of the 80 indefinite instances, as shown by three statistics computed by \code{Gurobi}. \code{ObjBound} is the best theoretical lower bound on the global minimum, \code{ObjVal} is the objective value of the incumbent solution, thus an upper bound on the global minimum, and \code{Gap} is the relative optimality gap defined as $\code{Gap} = \tfrac{|\mathtt{ObjVal}- \mathtt{ObjBound}|}{|\mathtt{ObjVal}|}$. In 77 out of the 80 instances, $\mathtt{ObjVal} = \mathtt{ObjBound}$, and $\mathtt{Gap} = 0$. For the other 3 cases, \code{Gurobi} returned a local solution after reaching a $\mathtt{Gap}$ smaller than the prescribed relative MIP optimality gap of $10^{-5}$. \textbf{\underline{Bottom line:}  As expected, running time explodes for indefinite instances and increasing dimension.}
\end{example}

 \begin{figure}[!htbp]
        \centering
\includegraphics[width=0.9\textwidth]{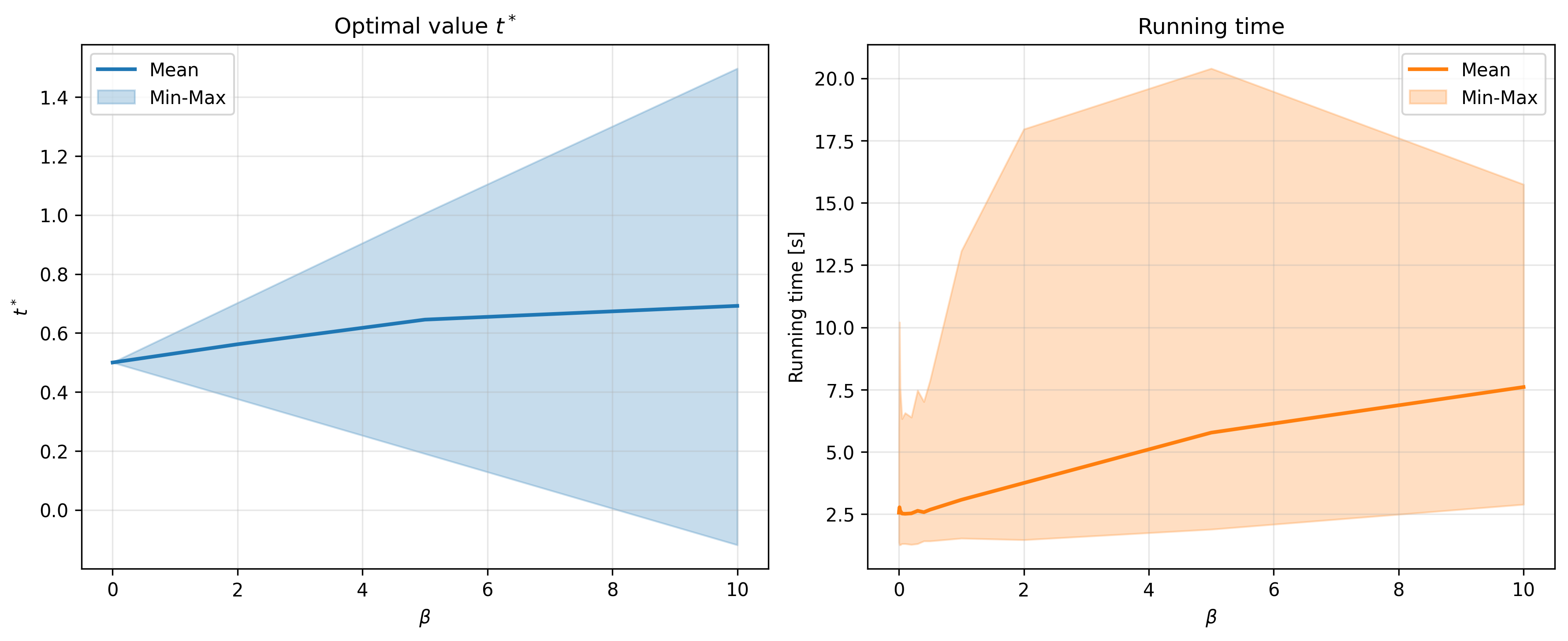}
\caption{GOE perturbation model with $\Qb_{\mathrm{nom}} = \Ob$. Optimal value $t^*$ and running time in seconds for different choices of $\beta$, for fixed $\varepsilon=0.1$, $\theta=0.05$, $n=10$, $N=20$.}\label{fig:Qnom=0-goe-beta-seeds-sweep}
 \end{figure}

\begin{table}[H]
\centering
\scalebox{0.8}{\begin{tabular}{|c|cc|cc|}
\hline
\multirow{2}{*}{$n$}
& \multicolumn{2}{c|}{$\Qb_{\mathrm{nom}} = \Ob$, $\{\widehat{\Qb}_i\}_{i\in[N]}\sim \mathcal W_n(\Ib,n)$}
& \multicolumn{2}{c|}{$\Qb_{\mathrm{nom}} = \Ob$,  $\{\widehat{\Qb}_i\}_{i\in[N]}\sim \operatorname{GOE}(n)$} \\
\cline{2-5}
& $t^*$
& running time
& $t^*$
& running time\\
\hline
$5$  & $8.99 \pm (1.86)$ & $0.32s \pm (0.08s)$ & $1.80 \pm (0.50)$ & $1.04s \pm (0.52s)$ \\
$10$ & $7.49 \pm (0.80)$ & $0.80s \pm (0.09s)$ & $0.64 \pm (0.20)$ & $8.63s \pm (6.69s)$\\
$15$ & $7.09 \pm (0.72)$ & $1.26s \pm (0.12s)$ & $0.38 \pm (0.18)$ & $215.76s \pm (135.79s)$\\
$20$ & $6.75 \pm (0.28)$ & $2.34s \pm (0.49s)$ & $0.08 \pm (0.24)$ & $3316.58s \pm (2393.10s)$\\
\hline
\end{tabular}}
\caption{Optimal value $t^*$ and running times for Wishart Ensemble and GOE perturbation models with $\Qb_{\mathrm{nom}} = \Ob$, for different $n\in\{5,10,15,20\}$ and fixed $\varepsilon=0.1$, $\theta=0.05$, $N=20$. The optimal value $t^*$ and the running time are given as mean $\pm$ (standard deviation) across the 20 seeds.}
\label{tab:n-Wishart-GOE-no-Qnom}
\end{table}

\subsection{With nonzero nominal matrix}
\subsubsection{Synthetic data}
\begin{example}[Synthetic nominal matrices] We repeated the experiments from the previous section with $\Qb_{\mathrm{nom}}\neq \Ob$ with positive definite nominal matrices for the Wishart Ensemble perturbation model and indefinite matrices for the GOE perturbation model. For that, we generated symmetric matrices $\Ub$ entrywise from the uniform distribution on $(0,1)$, then we put $\Qb_{\mathrm{nom}} =\Ub^{\top}\Ub$ in the Wishart Ensemble case and $\Qb_{\mathrm{nom}} =\Ub$ in the GOE case. Varying the parameters has the same effect. Table~\ref{tab:n-Wishart-GOE-with-Qnom} reports the same comparison of optimal values $t^*$ and running times across dimensions $n$ for both perturbation models, now with a nominal matrix $\Qb_{\mathrm{nom}} \neq \Ob$. As in the case with zero nominal matrix, the running times grow steeply with $n$ for the indefinite GOE instances, while the positive definite Wishart instances can be solved fast. The big-$M$ for Wishart instances was on average $M = 167$; for GOE instances we found $M=62$. \textbf{\underline{Bottom line:} Even though the eigenvalues of the unscaled (even multiplied by $\beta$) matrices dominate the eigenvalues of nominal matrices, the running times are shorter than in the case with zero nominal matrix.}
\end{example}

\begin{table}[H]
\centering
\scalebox{0.8}{\begin{tabular}{|c|cc|cc|}
\hline
\multirow{2}{*}{$n$}
& \multicolumn{2}{c|}{$\Qb_{\mathrm{nom}} \neq \Ob$, $\{\widehat{\Qb}_i\}_{=1}^N\sim \mathcal W_n(\Ib,n)$}
& \multicolumn{2}{c|}{$\Qb_{\mathrm{nom}} \neq \Ob$,  $\{\widehat{\Qb}_i\}_{=1}^N\sim \operatorname{GOE}(n)$} \\
\cline{2-5}
& $t^*$
& running time
& $t^*$
& running time\\
\hline
$5$  & $10.39 \pm (1.83)$ & $0.35s \pm (0.10s)$ & $2.29 \pm (0.53)$ & $0.88s \pm (0.43s)$ \\
$10$ & $10.19 \pm (0.90)$ & $0.92s \pm (0.28s)$ & $1.15 \pm (0.23)$ & $6.11s \pm (3.16s)$\\
$15$ & $10.85 \pm (0.80)$ & $1.50s \pm (0.45s)$ & $0.85 \pm (0.19)$ & $147.71s \pm (100.89s)$\\
$20$ & $11.77 \pm (0.40)$ & $2.88s \pm (1.01s)$ & $0.58 \pm (0.22)$ & $1287.28s \pm (1217.51s)$\\
\hline
\end{tabular}}
\caption{Optimal value $t^*$ and running times for Wishart Ensemble and GOE perturbation models with $\Qb_{\mathrm{nom}} \neq \Ob$, for different $n\in\{5,10,15,20\}$. The optimal value $t^*$ and the running time are given as mean $\pm$ (standard deviation) across the 20 seeds.}
\label{tab:n-Wishart-GOE-with-Qnom}
\end{table}

\subsubsection{Real-world instances}
\noindent Now, we compare the solutions yielded by deterministic StQPs with those from DRCCStQPs as follows. It can be seen that \eqref{StQP} can also be written equivalently in \emph{epigraphic form}: \sloppy$
\min\{t: \x^{\top}\Qb\x \leq t, \x \in \Delta, t \in \mathbb R\}
$. The \eqref{DRCCStQP} enforces a chance constraint on the epigraphic constraint. 
\begin{example}[Mathematical ecology]\label{ex:Bomze_DeKlerk} Consider the setting of \cite[Example~5.3]{bomze2002solving}: Mathematical model 410A from \citep{vickers1988patterns}, where the objective is to minimize $\x^{\top}\Qb_{\mathrm{nom}}\x$ subject to $\x \in \Delta$ with $\Qb_{\mathrm{nom}}=-\overline{\Qb}$, for the $\overline{\Qb}$ specified in \cite[Example~5.3]{bomze2002solving}. This is an indefinite matrix with $\lambda_{\mathrm{min}}(\Qb_{\mathrm{nom}}) = -63.1$, $\lambda_{\mathrm{max}}(\Qb_{\mathrm{nom}}) = 33.12$. The global solution is given by $\x_{\mathrm{nom}}^* = (0, \tfrac{1}{3}, \tfrac{1}{3}, \tfrac{1}{3}, 0)^{\top}$ and the optimal value is $\upsilon(\Qb_{\mathrm{nom}}) = -16\tfrac{1}{3}$. Now, we put $\varepsilon = 0.1$, $\theta = 0.05$, $\beta = 5$, and $N=20$ as in Table~\ref{tab:baseline-parameters}. Then, we sample $\{\widehat{\Gb}_i\}_{i\in[N]}\sim\operatorname{GOE}(n)$ and put $\widehat{\Qb}_i = \Qb_{\mathrm{nom}} + \beta \widehat{\Gb}_i$. We compute $M=69$ according to \eqref{eqn:big-M}. The DRCC solution is $\x^*_{\mathrm{DRCC}}\approx (0, 0.21, 0.35, 0.44, 0)^{\top}$, shifting weight from the second entry to the fourth entry. Compared to the nominal case, it yields an objective value of $(\x^*_{\mathrm{DRCC}})^{\top}\Qb_{\mathrm{nom}}\x^*_{\mathrm{DRCC}} \approx -15.97$, which means a relative increase in the objective value of $2.22\%$. We construct an i.i.d. unseen sample $\{\widehat{\Qb}_i^{\prime}\}_{i\in[N^{\prime}]}$ with $N^{\prime} = 10^4$ and compute the violation probabilities. For the DRCC solution, we obtain a violation probability of $7.25\%$. We compute the worst-case distribution $\mathbb P^*$ according to Section~\ref{subsec:construction-wc-distr} %\eqref{eqn:optimal-transport-plan} 
and report the quadratic forms $f(\x,\Qb) = \x^{\top}\Qb\x$ on Table~\ref{tab:real-world-math-model_n=5}. The nominal solution performs better on the nominal objective, while the DRCC solution performs better on the other objectives. \textbf{\underline{Bottom line:} Adding the DRCC increases the nominal cost, but reduces the out-of-sample constraint violation probability}. 
\end{example}

\begin{table}[H]
\centering
\scalebox{0.8}{\begin{tabular}{|c|ccc|}
\hline
$f(\x,\Qb), \text{for } \Qb =$ & $\Qb_{\mathrm{nom}}$ & $\mathbb E_{\widehat{\mathbb P}_N}[\widetilde \Qb]$ & $\mathbb E_{\mathbb P^*}[\widetilde \Qb]$\\
\hline
$\x = \x_{\mathrm{nom}}$ & $-16.33$ &  $-16.57$& $-16.52$\\
$\x = \x_{\mathrm{DRCC}}$ & $-15.97$ &  $-16.71$& $-16.66$\\
\hline
\end{tabular}}
\caption{Evaluations of the quadratic form $f(\x,\Qb) = \x^{\top}\Qb\x$ for different matrices and solution vectors in Example~\ref{ex:Bomze_DeKlerk}.}
\label{tab:real-world-math-model_n=5}
\end{table}

\begin{example}[Minimum variance portfolio]\label{ex:mvp} Consider the following minimum variance portfolio optimization problem, which can be cast as an StQP: $\min_{\x \in \Delta} \x^{\top}\sSigma\x$. Here, $\sSigma$ is the covariance matrix of stocks and the variables represent the fraction of the portfolio allocated to a stock, which must sum to one. Note that in this model, short-selling is not allowed. We took the  S\&P 500 Top 10 companies by market capitalization as of September 17, 2026. These companies accounted for approximately 38\% of the total market capitalization. Using the market data from \code{Yahoo! Finance's API} freely available in \code{Python}, we obtained the daily closing prices for those companies from January 1, 2015 to December 31, 2024. All $n=10$ companies were active during this period.  {We used the periods 2015--2022 for training and the periods 2023--2024 for testing, which resulted in 2013 observations for training and 
502 observations for testing.} Following~\citep{ledoit2004well}, we computed  {for the training period 2015--2022} the sample covariance matrix $\Sb$, the shrinkage target $\Fb = \tfrac{\operatorname{tr}(\Sb)}{n}\Ib$ and the shrinkage intensity $\delta$ and put $\sSigma = (1-\delta) \Sb + \delta \Fb$. In our case, $\delta \approx  {0.8\%}$. The matrix $\sSigma$ was positive definite with $\lambda_{\mathrm{min}}(\sSigma) \approx  {6}\cdot 10^{-6}$ and $\lambda_{\mathrm{max}}(\sSigma) \approx 3\cdot 10^{-3}$. For any eigenpair $(\lambda_i, \v_i)$ of $\sSigma$ we have $\|\sSigma\v_i - \lambda_i \v_i\|_2 \leq 10^{-17}$. We put $\Qb_{\mathrm{nom}} = 10^4\cdot\sSigma$ and solve the StQP with global solution given by $\x^*_{\mathrm{nom}}$ and optimal value $\upsilon(\Qb_{\mathrm{nom}}) = {1.5581}$. We aim to minimize the Value-at-Risk of the stock portfolio variance. We put $\varepsilon = 0.1$, $\theta = 0.1$, $\beta = 5$, $M = 5$ and $N=20$. Then, we sample $\{\widehat{\Gb}_i\}_{i\in[N]}\sim\operatorname{GOE}(n)$ and put $\widehat{\Qb}_i = \Qb_{\mathrm{nom}} + \beta \widehat{\Gb}_i$. Compared to the nominal case, the DRCC solution $\x^*_{\mathrm{DRCC}}$ yields an objective value of $(\x^*_{\mathrm{DRCC}})^{\top}\Qb_{\mathrm{nom}}\x^*_{\mathrm{DRCC}} \approx  {2.3473}$, which means a relative increase in the objective value of $ {50.66\%}$. The nominal solution puts most of the weight on the ticker ``BRK-B," whereas the DRCC solution favors a more balanced portfolio; see Table~\ref{tab:tickers}. {We build a nominal matrix $\Qb_{\mathrm{nom, test}}$ using the testing periods 2023--2024. From that, } we construct an  i.i.d. unseen sample $\{\widehat{\Qb}_i^{\prime}\}_{i\in[N^{\prime}]}$ with $N^{\prime} = 10^4$ and compute the violation probabilities.  {For the nominal solution, we obtain a violation probability on the unseen periods of $42.91\%$, whereas the DRCC solution achieves a violation probability of $4.64\%$.} We compute the worst-case distribution $\mathbb P^*$ according to Section~\ref{subsec:construction-wc-distr} and report the quadratic forms $f(\x,\Qb) = \x^{\top}\Qb\x$ on Table~\ref{tab:real-world-math-model}. The nominal solution performs better on all objective functions, at the cost of poor generalisation and violation probabilities.
\end{example}

\begin{table}[H]
\centering
\resizebox{\textwidth}{!}{
\begin{tabular}{|l|rrrrrrrrrr|}
\hline
 & AAPL & AMZN & AVGO & BRK-B & GOOG & GOOGL & META & MSFT & NVDA & TSLA \\
\hline
$\x_{\mathrm{nom}}$ & 2.78\% & 10.48\% & 0.00\% & 76.97\% & 1.24\% & 8.27\% & 0.26\% & 0.00\% & 0.00\% & 0.00\% \\
$\x_{\mathrm{DRCC}}$ & 26.96\% & 12.07\% & 6.14\% & 10.52\% & 0.00\% & 29.22\% & 1.49\% & 10.79\% & 0.00\% & 2.80\% \\
\hline
\end{tabular}}
\caption{Portfolio weights for nominal and DRCC solutions for S\&P 500 Top 10 companies in Example~\ref{ex:mvp}.}
\label{tab:tickers}
\end{table}

\begin{table}[H]
\centering
\scalebox{0.8}{\begin{tabular}{|c|cccc|}
\hline
$f(\x,\Qb), \text{for } \Qb =$ & $\Qb_{\mathrm{nom}}$ & $\mathbb E_{\widehat{\mathbb P}_N}[\widetilde \Qb]$ & $\mathbb E_{\mathbb P^*}[\widetilde \Qb]$&  {$\Qb_{\mathrm{nom, test}}$}\\
\hline
$\x = \x_{\mathrm{nom}}$&  {$1.56$} &  {$1.02$} &   {$1.12$} &  {$0.77$}\\
$\x = \x_{\mathrm{DRCC}}$ &  {$2.35$} &   {$1.73$}&  {$1.83$} &  {$1.48$}\\
\hline
\end{tabular}}
\caption{Evaluations of the quadratic form $f(\x,\Qb) = \x^{\top}\Qb\x$ for different matrices and solution vectors in Example~\ref{ex:mvp}.}
\label{tab:real-world-math-model}
\end{table}

\section*{Conclusion}

\noindent In this paper, we discussed the DRCCStQP with Wasserstein ambiguity, a quadratic DRCCP with an uncertain data matrix. We established a connection between the CCStQP and the DRCCStQP: the CCStQP is the Value-at-Risk of the random variable $\x^{\top}\widetilde\Qb\x$, and the DRCCStQP is its robustified version, the worst-case Value-at-Risk. We demonstrated that this nonconvex problem admits an exact reformulation as a mixed-integer nonlinear optimization problem for a special choice of matrix norm. In particular, the DRCCStQP can also be reformulated as a deterministic StQP for a specific setting. In addition, we established a hard feasibility ceiling for the Wasserstein radius. Finally, we investigated numerical performance and how different model parameters affect the optimal value and running time. Empirical findings were mostly in line with theory. However, surprisingly, running times are shorter when the nominal matrix is nonzero.

\section*{Acknowledgements}
\noindent The authors thank Abdel Lisser, Heng Zhang, Ralf Werner and Francesca Maggioni for stimulating discussions and valuable suggestions. The authors are indebted to VGSCO and PGMO for financial support that enabled the presentation of this work at various major conferences: ECSO-CSM 2024, PGMODays (2024 and 2025), and ICCOPT 2025. The authors benefited from feedback from delegates at these meetings, which helped improve the quality of this paper. Furthermore, VGSCO support enabled S.~Küçükyavuz to visit the University of Vienna, making this collaboration possible. S.~Küçükyavuz is supported, in part, by ONR Grant \#N00014-22-1-2602.

\section*{Conflict of interest}

\noindent The authors declare no conflict of interest.

\bibliography{references}
\end{document}